\documentclass[10pt]{amsart}
\usepackage{amsthm,amssymb,latexsym,amsmath}
\usepackage{color}
\usepackage{fancyvrb} 
\usepackage{graphicx}

\usepackage[margin=2.9cm]{geometry}
\usepackage{epstopdf} 
\def \C {\mathbb{C}}

\def \W {{\bf W}}
\def \E {{\bf E}}
\def \h {{\bf h}}

\def \ddd {\mathcal{D}}
\def \dps {\displaystyle}

\def \fol {{\mathcal F}}
\def \F {{\mathcal F}}

\def \G {{\mathcal G}}
\def \folt {\widetilde{\mathcal{F}}}

\def \M {{\bf M}}
\def \N {\mathbb{ N}}
\def \nn {\mathcal{N}}

\def \oo {\mathcal{O}}
\def \aa {\mathcal{A}}

\def \P {\mathbb{P}}
\def \pn {\mathbb{P}^n}

\def \Q {\mathbb{ Q}}

\def \R {\mathbb{R}}
\def \zz {\mathbb{Z}}

\def \sing {{\rm Sing}}

\def \tt {\mathcal{T}}

\def \Z {{\bf Z}}
\newtheorem{proposition}{Proposition}[section]
\newtheorem{definition}[proposition]{Definition}
\newtheorem{corollary}[proposition]{Corollary}
\newtheorem{lema}[proposition]{Lemma}
\newtheorem{theorem}[proposition]{Theorem}
\newtheorem{example}[proposition]{Example}
\newtheorem{remark}[proposition]{Remark}
\begin{document}

\title[A   Van den Essen type formula  and applications]{ A higher-dimensional  Van den Essen type formula for projective foliations and applications}

\author{Maur\'icio Corr\^ea, Gilcione Nonato Costa}
\address{Maur\'icio Corr\^ea  \\ 
Universit\`a degli Studi di Bari, 
Via E. Orabona 4, I-70125, Bari, Italy
}
\email[M. Corr\^ea]{mauricio.barros@uniba.it,mauriciomatufmg@gmail.com } 
\address{Gilcione Nonato Costa \\
ICEx - UFMG \\
Departamento de Matem\'atica \\
Av. Ant\^onio Carlos 6627 \\
30123-970 Belo Horizonte MG, Brazil} \email{gilcione@mat.ufmg.br}


\begin{abstract}
Let  $\mathcal{F}$ be a one-dimensional holomorphic foliation on $\mathbb{P}^n$  such that 
$\W\subset \text{Sing}(\mathcal{F})$, where 
$\W$ is a smooth complete intersection variety.  We determine and compute the variation of the Milnor number \( \mu(\mathcal{F}, \W) \) under blowups, which depends on the vanishing order of the pullback foliation along the exceptional divisor, as well as on numerical and topological invariants of \( \W \).
 This represents a higher-dimensional version of Van den Essen's formula for projective foliations of dimension one. As an application, we obtain a lower bound for the Milnor number of the foliation. Also, we use this formula to show that for a foliation on $\mathbb{P}^n$ that is singular along a smooth curve, there exists a finite number of blow-ups with centers on smooth curves such that the induced foliation has multiplicity equal to 1
  and that for generic points of the curves in the final stage, the singularities are elementary. Moreover, we obtain a bound on the maximum number of blow-ups needed to resolve the foliation, depending on the numerical and topological invariants of the curve.
\end{abstract}

\dedicatory{In memory of Celso dos Santos Viana}

\maketitle

\section{Introduction}

\par 
In his celebrated work \cite{Bott}, Bott provided a method to compute residues of global holomorphic vector fields along non-degenerate and non-isolated singularities. In general, determining the residue for degenerate singularities in the case of meromorphic vector fields is challenging, except for isolated singularities, as shown by Baum and Bott in \cite{bb}. 

F. Bracci and T. Suwa established in \cite{BS} that Baum-Bott indices continuously vary under smooth deformations of holomorphic foliations, and in particular, such residues/indices can be computed via deformation. Therefore, following Bracci and Suwa, we can, in particular, define and compute the Milnor number of a one-dimensional foliation along a smooth subvariety of high codimension contained in the singular set as follows. Consider a one-dimensional holomorphic foliation $\mathcal{F}_0$ on a complex manifold $\M_0$ induced by a global section of $v_0\in H^0(X, T\M_0\otimes L)$, for some fixed line bundle $L$, such that its singular locus $\text{Sing}(\mathcal{F}_0)$ contains a smooth subvariety $\W_0$ of codimension $d \geq 2$. Now, 
let $\fol_t$ be a generic holomorphic
deformation of $\fol_0$, for $t\in D(0,\epsilon)$, with $\epsilon$ sufficiently small
such that $\fol_t$ is induced by a section of $v_t\in H^0(X, T\M_0\otimes L)$,
$\dps\lim_{t\to0}v_t=v_0$ and
$\sing(\fol_t)=\{p_1^t,\ldots,p_{m_t}^{t}\}$, where each $p_j^t$ is an
isolated closed point. Then  \textit{Milnor number} $\mu(\fol_0, \W_0)$  of $\fol_0$ along  $\W_0$ is given by 
$$\mu(\fol_0,\W_0)=\dps\sum_{\dps\lim_{t\to0} p_j^t\in\W_0}\mu(\fol_t,p_j^t),$$ where $\mu(\fol_t,p_j^t)$ is the usual Milnor number for isolated singularities. In \cite{AG}, we computed the Milnor number for the case where $\W_0$ is a non-dicritical component. Here, we extend that result to the case where $\W_0$ is a dicritical component of $\text{Sing}(\fol_0)$.

It is natural to ask how the Milnor number varies under certain maps that modify the foliation. In \cite{AGR},  the authors show that   $\mu(\fol_0,\W_0)$  on a three-dimensional manifold $\M_0$ remains invariant under topological equivalences $C^1$. In this work, our focus is on determining the \textit{Milnor number} $\mu(\fol_0, \W_0)$ and computing its variation under\textit{ blow-ups} for foliations on projective spaces.

\par In order to present our first results we need to fix the following notation: $\W_0:=Z(f_1,\ldots,f_d)$ will be a smooth complete intersection variety on $\M_0=\pn$ where $f_j$ is a reduced polynomial with $k_j=\deg(f_j)$ for $j=1,\ldots,d$.
Let $\tt_{\W_0}$ and $\nn:=\nn_{\W_0}$ be tangent and normal bundles of
$\W_0$ in $\M_0$ and with their total Chern classes $c(\tt_{|W_0})=\sum
\tau_i^{(d)}\h^i$ and $c(\nn)=\sum \sigma_i^{(d)}\h^i$,
respectively, where $\h$ is the hyperplane class of $\pn$. Consider 
$$\mathcal{W}^{(d)}_{\delta}:=\mathcal{W}^{(d)}_{\delta}(k_1,\ldots,k_{d})=\sum_{i_{1}+\ldots+i_{d}=\delta}k^{i_{1}}_{1}\ldots k^{i_{d}}_{d},$$
 the complete symmetric function
of degree $\delta$ in $d$ variables at the multi-indices
$(k_1,\ldots,k_d)$. 
Now, let $\pi_1:\M_1\to\M_0$ be the blowup of $\M_0=\pn$ along $\W_0$, with the exceptional
divisor $\E_1=\pi_1^{-1}\big(\W_{0}\big)$. The kernel
$\nu(\fol_0,\W_0,\varphi_a)$ is defined as follows
$$
\nu(\fol_0,W_0,\varphi_a)=-\deg(\W_0)\sum_{|a|=0}^{n-d}\sum_{m=0}^{n-d-|a|}(-1)^{\delta_{|a|}^{m}}\frac{\varphi_{a}^{(m)}(\ell)}{m!}(k-1)^{m}\sigma_{a_1}^{(d)}\tau_{a_2}^{(d)}\mathcal{W}_{\delta_{|a|}^{m}}^{(d)},
$$
where $k=\deg(\fol_0)$, $a=(a_1,a_2)\in\zz^2$, $|a|=a_1+a_2\ge0$,
$\delta_{|a|}^{m}:=n-d-|a|-m$; and $\ell$ is given by
$$\ell=\left\{\begin{array}{ll}
              m_{\E_1}\big(\pi_1^*\fol_0\big),& \mbox{ if } {\W_0 \mbox{ is non-dicritical}}\cr
              m_{\E_1}\big(\pi_1^*\fol_{0}\big)-1,&\mbox{ if } {\W_0 \mbox{ is dicritical}},
              \end{array}\right. $$
where   $m_{\E_1}\big(\pi_1^*\fol_0\big)$  denotes the vanishing order of the pullback foliation  $\pi_1^*\fol_0$  at $\E_1$ and the function $ \varphi_{a}(x):=x^{n-d-a_2}(1+x)^{d-a_1}$ with $\varphi_{a}^{(m)}(x)=\dps\frac{d^m}{dx^m}\varphi_{a}(x).$

\begin{theorem}\label{theorem1}
Let $\fol_0$ be a holomorphic foliation by curves on $\pn$, with
$n\geq 3$, of degree $k$. Suppose that the singular set of $\fol_0$ is the disjoint union of a smooth scheme-theoretic
complete intersection subvariety $\W_0$ of pure codimension $d\geq
2$, and closed points $p_1,\ldots,p_s$. Consider the blow-up
$\pi_1:\M_1\to\M_0$ centered on $\W_0$ being
$\E_1=\pi_1^{-1}\big(\W_0\big)$ the exceptional divisor and $\fol_1$ the strict transform foliation obtained from $\fol_{0}$ via
$\pi_1$.  Then
\begin{enumerate}
 \item [(a)] $\displaystyle\sum_{i=1}^{s}\mu(\fol_0,p_i)=\dps\sum_{i=0}^{n}k^i+\nu(\fol_0,W_0,\varphi_a)-N(\fol_0,\aa_{\W_0}),$
 \item [(b)] $\mu(\fol_0,\W_0) = -\nu(\fol_0,\W_0,\varphi_a)+N(\fol_0,\aa_{\W_0})\ge -\nu(\fol_0,\W_0,\varphi_a),$
 \item[(c)] $\mu(\fol_1,\bigcup_{i}\W_i^{(1)})=\mu(\fol_0,\W_0)+\nu(\fol_0,W_0,\vartheta_a)$
\end{enumerate}
where $N(\fol_0,\aa_{\W_0})$  is the number  of embedding closed points associated to $\W_0$, counted with 
multiplicities,  $\W_i^{(1)}\subset \E_1$ is each connected component of $\sing(\fol_1)$ and
$$\vartheta_{a}(x)=\varphi_{a}(x)+
x^{n-d-a_2-1}\big(1-(1+x)^{d-a_1}\big).$$
\end{theorem}

This is a higher-dimensional version of \textit{Van den Essen's formula}; see \cite[Theorem 1.3]{Van-den-Essen} and \cite[Appendix]{Mattei-Moussu}. Theorem \ref{theorem1} provides a lower bound for the Milnor number 
$$\mu(\fol_0,\W_0) \geq   -\nu(\fol_0,\W_0,\varphi_a)=\deg(\W_0)\sum_{|a|=0}^{n-d}\sum_{m=0}^{n-d-|a|}(-1)^{\delta_{|a|}^{m}}\frac{\varphi_{a}^{(m)}(\ell)}{m!}(k-1)^{m}\sigma_{a_1}^{(d)}\tau_{a_2}^{(d)}\mathcal{W}_{\delta_{|a|}^{m}}^{(d)}$$
and also  calculates its variation under blow-up. 

Item (c) remains valid even when $\W_0$  has dimension zero.

In the next part of this paper, we explore a holomorphic foliation $\fol_0$ defined on $\mathbb{P}^n$, where its singular set contains a smooth regular curve $\W_0$. We start with $\M_0 = \mathbb{P}^n$ and consider a sequence of blow-ups $\pi_j: \M_j \to \M_{j-1}$ centered, for each $j \geq 1$, along a component $\W_{j-1} \subset \text{Sing}(\mathcal{F}_{j-1})$, where $\mathcal{F}_j$ is the strict transform obtained from $\mathcal{F}_{j-1}$ under $\pi_j$, and $\E_j$ is the exceptional divisor. In short, we denote this sequence by $\{\pi_j, \M_j, \W_j, \mathcal{F}_j, \E_j\}.$
Generally, when we restrict the singular set of $\mathcal{F}_j$ to the exceptional divisor $\E_j$, it comprises new components and potentially isolated closed points. However, these new curves can be categorized into two main types: those that are homeomorphic to $\W_0$ and those that are homeomorphic to $\mathbb{P}^{n-2}$. As a result, we present the following theorem, which efficiently determines $\mu(\fol_i, \W_i)$ under natural hypotheses for $i\ge0$.
\begin{theorem}\label{theoremB} Let $\fol_0$ be a one-dimensional
holomorphic foliation defined on $\P^n$  such
that $\W_0\subset\sing(\fol_0)$ where $\W_0$ is a smooth curve of degree $\deg(\W_0)$ and Euler characteristic $\chi(\W_0)$.

If there is a blow-up sequence $\{\pi_j,\M_j,\W_j,\fol_j,\E_j\}$ where $\W_j$ is homeomorphic to $\W_{j-1}$ and $\pi_j(\W_j) = \W_{j-1}$ for $j \geq 1$, then
\begin{enumerate}
    \item [(a)] \begin{align*}\displaystyle   
    \mu(\fol_j,\W_j)=& (1+\ell_{j+1})^{n-1}\chi(\W_0)-\ell_{j+1}^2(1+\ell_{j+1})^{n-2}\frac{\Lambda_0^{(n)}}{(n-1)^j}+\cr
                 &+(n\ell_{j+1}+1)(1+\ell_j)^{n-2}\bigg[(k-1)\deg(\W_0)-\Lambda_0^{(n)}\sum_{i=1}^j\frac{\ell_i}{(n-1)^i}\bigg]+N(\fol_j,\aa_{\W_j})
\end{align*}
where $\Lambda_0^{(n)}:=(n+1)\deg(\W_0)-\chi(\W_0)$,  $\ell_i=m_{\E_i}\big(\pi_i^*\fol_{i-1}\big)$ for $i=1,\ldots,j+1$ and $N(\fol_{j},\aa_{\W_{j}})$ the number of embedding closed points associated with $\W_j$, counted with multiplicities.

\item [(b)] \begin{align*}
 \mu(\fol_{j+1},\bigcup_{i}\W_i^{(j+1)})&=\mu(\fol_j,\W_j)+\chi(\W_0)\bigg(\sum_{i=0}^{n-2}(1+\ell_{j+1})^i-(1+\ell_{j+1})^{n-1}\bigg)\cr
 &+(\ell_{j+1}^2-\ell_{j+1})(1+\ell_{j+1})^{n-2}\frac{\Lambda_0^{(n)}}{(n-1)^j}\cr
 &-(n\ell_{j+1}-n+2)(1+\ell_{j+1})^{n-2}\bigg((k-1)\deg(\W_0)-\Lambda_0^{(n)}\sum_{i=1}^j\frac{\ell_i}{(n-1)^i}\bigg)
 \end{align*}
 \end{enumerate}
 where $\W_i^{(j+1)}\subset \E_{j+1}$ is each connected component of $\sing(\fol_{j+1})$.
 
 Here we assume that $\dps\sum_{i=\alpha}^{\beta}a_i=0$ for $\alpha < \beta$.
\end{theorem}
The map $\pi_j: \M_j\setminus\E_j\to \M_{j-1}\setminus\W_{j-1}$ is a biholomorphism which implies that the sequence $N(\fol_{j},\aa_{\W_{j}})$ is non-increasing. 
Since $\mu(\fol_j, \W_j)$ is a natural number, the order $m_{\W_i}(\fol_i)$ (the order of vanishing of the foliation $\fol_i$ along the curve $\W_i$) typically increases by one during a blow-up sequence. This increase is expected to be $m_{\W_0}(\fol_0) + 1$. By Theorem \ref{theoremB}, we conclude that the order of annulment $\ell_i = m_{\E_i}(\pi_i^*\fol_{i-1})$ will be zero for sufficiently large $i$. This implies that after a certain point in the sequence of blow-ups, the foliation no longer increases in complexity, and the process stabilizes. Consequently, this result offers a new approach to extending Seidenberg's Theorem for foliations with non-isolated singularities.  
A well-known fact is that for $n=2$, the resolution Theorem of Seidenberg
\cite{sein68} asserts that all the singularities of $\fol_i$ are elementary for $i$ large enough. This means that if
$p\in \sing(\fol_i)$ then $\fol_i$ is locally generated by a vector
field $X_k$ having a linear part with eigenvalues $\lambda_1$ and
$\lambda_2$ where ${\lambda_1}/{\lambda_2} \not\in \Q_{+}$ or at
least one eigenvalues is non-zero. 

As defined in \cite{CCS}, the component $\W_0$ will be called an \textit{absolutely isolated singularity} if all the components $\W_i$, which appear in this process, have the same dimension as $\W_0$. In \cite{CCS}, the desingularization theorem is proven when $\W_0$ is a non-dicritical absolutely isolated singularity. The non-dicritical condition is removed in \cite{rb}. However, it is unfortunate that a general birational desingularization theorem is not possible, as shown in \cite{FSFS}, where F. Sanz and F. Sancho presented an example of a vector field $X$ that  \textit{cannot be} desingularized by any sequence of blowups. For further details, refer to Example \ref{SSe} and Proposition \ref{propSS}. For complex 3-folds, Panazzolo \cite{Panazzolo} and McQuillan–Panazzolo \cite{McQuillan–Panazzolo} provide a non-birational desingularization (after performing smoothed weighted blow-ups) and Cano in \cite{Cano} proposes a desingularization approach that permits the use of formal, non-algebraic blow-up centers. Our next result states that we can proceed with a birational desingularization such that, for generic points of the curves in the desingularized model, the singularities are elementary.

\begin{theorem}
\label{theoremA} Let $\fol_0$ be a one-dimensional holomorphic
foliation defined on $\M_0=\P^n$ such that
$\W_0\subset\sing(\fol_0)$
where $\W_0$ is a smooth curve.

If there is a blow-up
sequence $\{\pi_i,\M_i,\W_i,\fol_i,\E_i\}$ such that $\W_i$ is a homeomorphic curve to $\W_{i-1}$ and $\pi_j(\W_j)=\W_{i-1}$, then there is a natural number $k$ such that $\W_i$ is an
elementary component of $\fol_i$ for $i\ge k$, for almost all points of $\W_i$. In particular,  $\fol_i$ is  generically log canonical along $\W_i$. for all $i\ge k$. 
\end{theorem}

From Theorem \ref{theoremA}, we can conclude that after a finite number of blow-ups along
homeomorphic curves to $\W_0$,  the foliation $\fol_k$ in the open set $U_k$
is described by the following vector field.

$$\dps X_k=x_1\bigg(\lambda_1^{(k)}+P_1(x)\bigg)\frac{\partial}{\partial x_1}+\bigg(x_1r_1(x_3)+x_2\lambda_2^{(k)}+P_2(x)\bigg)\frac{\partial}{\partial x_2}+P_3(x)\frac{\partial}{\partial x_3}$$
where $\lambda_i^{(k)}=\lambda_i^{(k)}(x_3)$ for $i=1,2$, the singular component $\W_k$ is defined
as $x_1=x_2=0$ and $X_k$ having a linear part with eigenvalues $\lambda_1^{(k)}$ and
$\lambda_2^{(k)}$ where ${\lambda_1^{(k)}}/{\lambda_2^{(k)}} \not\in \Q_{+}$,  for almost all points of $\W_k$, or at
least one of these eigenvalues is  not identically zero. 

As a consequence of the \ref{theoremB}, we obtain a bound on the maximum number of blowups needed to resolve the foliation 
which depends on the vanishing order of the pullback foliation along the exceptional divisor, as well as on numerical and topological invariants of the curve.

\begin{corollary}
       Let $\fol_0$ be a one-dimensional holomorphic
foliation defined on $\M_0=\P^n$ such that
$\W_0\subset\sing(\fol_0)$
where $\W_0$ is a smooth curve. Then, 
the maximum number of blowups needed to resolve the foliation is
\begin{equation} 
\sum_{\ell=1}{^\ell_1}\lfloor \log_{n-1}\ell(1+\ell)^{n-2}(2\ell+1)(\Lambda_0^{(n)})^{2})\rfloor 
\end{equation}
$\Lambda_0^{(n)}=(n+1)\deg(\W_0)-\chi(\W_0)$ and   $\ell_1=m_{\E_1}\big(\pi_1^*\fol\big)$. 
\end{corollary}

In \cite{fss},  F. Sancho de Salas presented a similar theorem, but with a key difference: the case where $\W_0$ has \textit{codimension two} and is also a \textit{ absolutely isolated} component.  Theorem 1.3 generalizes the Sancho de Salas' result in the case of   foliations on $\P^n$, as \textit{different dimensional singularities} may emerge during the  desingularization process; see example \ref{examp11}. Despite the challenges posed, we establish that if such a sequence exists, then, starting from a certain index, the singular components become elementary, \textit{regardless of whether they are absolutely isolated or not}.

\subsection*{Acknowledgments}
MC is partially supported by the Universit\`a degli Studi di Bari and by the
 PRIN 2022MWPMAB- ``Interactions between Geometric Structures and Function Theories'' and he is a member of INdAM-GNSAGA;
he was  partially supported by CNPq grant numbers 202374/2018-1, 400821/2016-8 and  Fapemig grant numbers  APQ-02674-21, APQ-00798-18,  APQ-00056-20. 

\section{Preliminaries}

\subsection{Holomorphic foliations with non-isolated singularities}
Let $\fol_0$ be a one-dimensional holomorphic foliation on  $\mathbb{P}^n$, with $n \geq 3$, such that its singular set contains a
smooth subvariety $\W_0$ of pure codimension $d\ge2$. Given that $\W_0$ is smooth, for each closed point
$p\in\W_0$ there exists an open set $U_0$ of $p$ and a coordinate
system $z\in\C^n$ such that $\W_0\cap U$ can be defined as
$\{z_1=\ldots=z_{d}=0\}$. Therefore, in $U_0$ the foliation $\fol_0$
is described by the following vector field
\begin{equation}
\label{fol} X_0=\dps P_1(z)\frac{\partial}{\partial z_1}+\ldots+\dps
P_n(z)\frac{\partial}{\partial z_n}
\end{equation}
which we can write the local sections as
\begin{equation}
\label{for40}
        P_i(z) = \dps\sum_{|a|=m_i}z_1^{a_1}\cdots z_{d}^{a_{d}}P_{i,a}(z)
\end{equation}
where $a:=(a_1,\ldots,a_{d})\in \zz^d$ with $|a|:=a_1+\ldots+a_{d}$,
$a_i\ge0$ and for each $i$ at least one among the $P_{i,a}(z)$ does
not vanish at $\{z_{1}=\ldots=z_{d}=0\}$. The natural number $m_i$  will be called the multiplicity
of $P_i$ along $\W_0$ and denoted by $m_{\W_0}(P_i)$.
\begin{definition}\label{defm} The
\textit{multiplicity of $\fol_0$ along $\W_0$} is defined as follows
$$\mbox{m}_{\W_0}(\fol_0)={\rm min}\{m_1,\ldots,m_n\}.$$
\end{definition}

\begin{lema}\label{lemafp} Let $\fol_0$ be a holomorphic foliation by curves,
defined in the neighborhood of $p$ by the vector field
$$ X_p = \displaystyle\sum^{n}_{j=1}P_{j}(z)\frac{\partial}{\partial{z_{j}}}$$
as in (\ref{fol}) with $m_i=m_{W_0}(P_i)$. Then, by a linear change of coordinates,
$\fol_0$ may be described by the vector field 
$$ Y_p = \displaystyle\sum^{n}_{j=1}Q_{j}(w)\frac{\partial}{\partial{w_{j}}}$$
with $$m_{W_0}(Q_j)=\left\{\begin{array}{ll}
                            m^{\prime}_{\W_0}(\fol_0),&{\rm for}\quad j=1,\ldots,d\cr
                            m_{\W_0}(\fol_0),&{\rm for}\quad j=d+1,\ldots,n
                            \end{array}\right.$$
where $m^{\prime}_{\W_0}(\fol_0)={\rm min}\{m_1,\ldots,m_d\}$.
\end{lema}
\begin{proof} In fact, it is enough to consider the linear
transformation $w=Az$ where $A=(a_{ij})\in GL(n,\C)$ with $a_{ij}=0$
for $i=1,\ldots,d$ and $j=d+1,\ldots,n$. In this way, $B=(b_{ij}) =
A^{-1}$ has the same properties, that is, the subspace $W$ given by
$\{z_1=\ldots=z_d=0\}$ is an invariant set under this
transformation. Adjusting the coefficients $a_{ij}$, if necessary,
we can admit that
$$ \dot w_i=\left\{\begin{array}{ll}
    \dps \sum_{j=1}^{d}a_{ij}\dot z_j=\dps\sum_{j=1}^{d}a_{ij}P_j(Bw)=Q_i(w),& \mathrm{ for }\quad
    i=1,\ldots,d\\
    \dps \sum_{j=1}^{n}a_{ij}\dot z_j=\dps\sum_{j=1}^{n}a_{ij}P_j(Bw)=Q_i(w),& \mathrm{ for }\quad
    i=d+1,\ldots,n
    \end{array}\right.$$
having the required properties.
\end{proof}
\par Without loss of generality, we can assume that the vector field in (\ref{fol}) is such that
$$m_{\W_0}(P_j)=
\begin{cases}
m_1, & {\rm for}\ j=1,\ldots,d \\
m_n, & {\rm for}\ j=d+1,\ldots,n.
\end{cases}
$$
with $m_1\ge m_n=\mbox{m}_{\W_0}(\fol)$. 

From now on,  wee will consider the
$d\times d$ complex matrix
\begin{equation}
\label{defmc} \mathbf{A}_{X_0} = \dps\bigg[\frac{\partial
P_i}{\partial z_j}\bigg]_{1\le i,j\le d}
\end{equation}
\begin{definition}\label{def2}
The component $\W_0$ will  be referred to as  locally elementary of
$\fol_0$ if the matrix $\mathbf{A}_{X_0}$ has at least one nonzero
eigenvalue for almost all points $z\in \W_0\cap U$, where $U$ is an open set. On the other hand, if all the eigenvalues of $\mathbf{A}_{X_0}|_{\W_0}$ are identically zero, then $\W_0$ will be called a non-elementary component of $\sing(\fol_0)$.
\end{definition}
Clearly, if $m_{\W_0}(\fol_0)\ge 2$ then $\mathbf{A}_{X_0} \equiv 0$
for all $z\in\W_0\cap U_0$. Moreover, the linear part of the vector
field $X_0$ restricted to $\W_0$ has at most $d$  identically zero eigenvalues.

We have the following proposition
\begin{proposition}
The definitions (\ref{defm}) and (\ref{def2}) are independent of the
chosen coordinate system.
\end{proposition}
\begin{proof}
In fact, it is enough to consider the biholomorphism $\Phi: U \to
\C^n$ given by $w=\Phi(z)=(\Phi_1(z),\ldots,\Phi_n(z))$ such that
$\Phi_i(z)\equiv 0$ for all $z\in\W_0\cap U$, for $i=1,\ldots,d$.
Let us admit that $\fol_0$ is described in an other coordinate
system by the following vector field
$$ Y_0(w)=\dps Q_1(w)\frac{\partial}{\partial
w_1}+\ldots+\dps Q_n(w)\frac{\partial}{\partial w_n}$$ where each
$Q_i(w)$ vanishing along $\{w_1=\ldots=w_d=0\}$. Let $\mathbf{B}_{Y_0}$ be the $d\times d$ complex
matrix given by
$$ \mathbf{B}_{Y_0} = \dps\bigg[\frac{\partial Q_i}{\partial w_j}\bigg]_{1\le i,j\le d}.$$

Since $\Phi$ is a local biholomorphism, we have that
$det(D\Psi|_{\W_0})\ne0$ and by consequence
$w_i=\dps\sum_{j=1}^{d}z_j\phi_{i,j}(z)$ for $i=1,\ldots,d$ and the
$d\times d$ complex matrix
$$\mathbf{C}=\dps\big[\phi_{i,j}\big]_{1\le i,j\le d}$$
is not singular for all $z\in\W_0\cap U$. In the same way,
$z=\Psi(w)=\Phi^{-1}(w)$ and
$z_i=\dps\sum_{j=1}^{d}w_j\psi_{i,j}(w)$ for some functions
$\psi_{i,j}$ and also for $i=1,\ldots,d$. Thus, $X_0=\Phi^*(Y_0)$ and
$\mathbf{B}_{Y_0}=\mathbf{C}\cdot \mathbf{A}_{X_0}\cdot\mathbf{C}^{-1}$,
which concluded that the definition (\ref{def2}) is independent of
the coordinate system chosen. But,
$$P_i\circ \Psi(w)=\dps\sum_{|a|=m_i}z_1^{a_1}\cdots z_{d}^{a_{d}}P_{i,a}(z)|_{z=\Psi(w)}=\dps\sum_{|a|=m_i}w_1^{a_1}\cdots w_{d}^{a_{d}}\widetilde{P}_{i,a}(w)$$ with some $\widetilde{P}_{i,a}(w)|_{\W_0}\not\equiv0$
which results

$$Q_i(w)=\left\{\begin{array}{ll}\dps\sum_{j=1}^{d}\phi_{i,j}\circ \Psi(w)\cdot P_j\circ \Psi(w),&1\le i\le d\\
                            \dps\sum_{j=1}^{n}\frac{\partial \Phi_i}{\partial z_j}\circ \Psi(w)\cdot P_j\circ \Psi(w),& d< i\le n.
           \end{array}\right.$$
Let $q_i:=\mbox{m}_{\W_0}(Q_i)$ for all $i=1,\ldots,n$. Now, let us
suppose by absurd that
$\mbox{m}_{\W_0}(\fol_0)\ne\mbox{min}\{q_1,\ldots,q_n\}$. Then,
applying the same arguments for the vector field $Y_0$ with
$z=\Psi(w)$ we will get
$\mbox{min}\{m_1,\ldots,m_n\}\ne\mbox{m}_{\W_0}(\fol_0)$.
\end{proof}

\par Let $\pi_1: U_1 \to U_0$  be the blowup of an open set $U_0$  along at $\W_0\cap U_0$ ,  with $\E_1$ is the exceptional divisor and $\fol_1$ is the strict transform of the foliation on $U_1$.
In order to simplify the notation, we will denote
$I_n=\{1,2,\ldots,n\}$ and $J_d=\{2,\ldots,d\}$. In the chart
$\big((U_1)_1,\sigma_1( u)\big)$, with coordinates $u=(u_i)\in\C^n$ (in the similar manner
in the others charts $\bigg(\big(U_1\big)_j,\sigma_j(v)\bigg)$ such that
\begin{equation}\label{for111}
z=\sigma_1(u)=(u_1,u_1u_2,\ldots,u_1u_d,u_{d+1},\ldots,u_n)=(z_1,\ldots,z_d,z_{d+1},\ldots,z_n)
\end{equation}
 and the pull-back foliation $\pi_1^*\fol_0$ is described by the
following vector field
\begin{equation}
\label{equ5en} \ddd_{\pi_1^*\fol_0}=\dps \sum_{i\in
I_n\setminus J_d}u_1^{m_i}\big(Q_i(u)+u_1\widetilde{P_i}(u)\big)\frac{\partial}{\partial
u_i} + u_1^{m_1-1}\sum_{i\in
J_d}\big(G_i(u)+u_1\widetilde{P_i}(u)\big)\frac{\partial}{\partial
u_i}
\end{equation}
with
$$Q_i(u)= \dps
\sum_{|a|=m_i}u_2^{a_2}\cdots u_{d}^{a_{d}}
p_{i,a}(u_{d+1},\ldots,u_n),\quad G_i(u)=Q_i(u)-u_iQ_1(u),$$
where $p_{i,a}(u_{d+1},\ldots,u_n)=P_{i,a}(0,\ldots,0,u_{d+1},\ldots,u_n)$
for certain functions $\widetilde{P_i}$. For more details, see
\cite{AG} or \cite{GH}.

As usual, we will say that $\W_0$ is a \textit{non-dicritical} component  if the
exceptional divisor $\E_1$ is invariant by $\fol_1$, otherwise
$\W_0$ is a \textit{dicritical} component.

After a division of (\ref{equ5en}) by an adequate power of $u_1$, we
obtain the expressions of the vector field $X_1$ that generates the
induced foliation $\fol_1$.

In the situation where $\W_0$ is a non-dicritical component of
$\sing(\fol_0)$, is described in the following three cases.

\noindent Case (i) : $m_n + 1 = m_1$ and $G_{j}\not\equiv 0$ for some
$j\in I_d$.In this situation, equation (\ref{equ5en}) is divided by $u_1^{m_1-1}$,

\begin{equation}
\label{equ6en} X_1=\dps\sum_{i\in
I_n\setminus J_d}u_1^{m_i-m_1+1}\bigg(Q_i(u)+u_1\widetilde{P_i}(u)\bigg)\frac{\partial}{\partial
u_i} + \sum_{i\in
J_d}^{d}\bigg(G_i(u)+u_1\widetilde{P_i}(u)\bigg)\frac{\partial}{\partial u_i}.
\end{equation}

This case has been widely studied in \cite{MAGR}, \cite{AG}, \cite{toulouse} and
\cite{indices}.

\noindent Case (ii) : $m_n + 1 < m_1$. As in the case before, dividing (\ref{equ5en}) by $u_1^{m_n}$ we get
\begin{align}
\label{equ7en}
X_1=&\dps
u_1^{m_1-m_n}\bigg(Q_1(u)+u_1\widetilde{P_1}(u)\bigg)\frac{\partial}{\partial
u_1} +
u_1^{m_1-m_n-1}\sum_{i\in
J_d}\bigg(G_i(u)+u_1\widetilde{P_i}(u)\bigg)\frac{\partial}{\partial
u_2}\cr
&+\sum_{i=d+1}^n\bigg(Q_i(u)+u_1\widetilde{P_i}(u)\bigg)\frac{\partial}{\partial
u_i}
\end{align}
In this case, the leaves of $\fol_1$ contained
in $\E_1$ can be locally described as follows 
\begin{equation}
\label{ft2}\bigg\{\varphi(t,p)= (\varphi_1(t),\varphi_2(t),\ldots,\varphi_n(t), \varphi(0,p)=p=(p_i)_i,t\in D(0,\epsilon)\bigg\}\end{equation}
with $\varphi_1(t)\equiv0$ and $\varphi_i(t)\equiv p_i$ are constant functions for all $i\in J_d$. In other words, these leaves are generically transverse to fibers
$\pi_1^{-1}(q)$, $q\in\W_0$.

By other side, the singular set of $\fol_1$ when it is restricted to $\E_1$ is given by the following $n-d$ equations 
\begin{equation}\label{sst2}
Q_i(0,u_2,\ldots,u_n)=0,\quad \mathrm{for} \quad i=d+1,\ldots,n
\end{equation}
which  can be generically solved for $n-d$ unknown variables $u_{d+1},\ldots, u_n$ as follows 
\begin{equation}
\label{defce}
\W_i=\{u\in\E_1| u_k=\psi_k(u_2,\ldots,u_d),  k>d\}.
\end{equation}
Thus, each singular component $\W_i$ has dimension equal to $d-1$ and is homeomorphic to $\P^{d-1}=\big\{ [u_1:u_2:\ldots :u_d]\big\}$.  Note that $\W_i$ can be given by the graph of the function $\Psi:\P^{d-1}\to \E_1$ locally defined as 
\begin{equation}\label{for112}
\Psi(u_1,u_2,\ldots,u_d)=(u_1,u_2,\ldots,u_d,\psi_{d+1}(u_2,\ldots,u_d),\ldots,\psi_n(u_2,\ldots,u_d)).
\end{equation}
with $u_1=0$ in this chart.

\noindent Case (iii) : $m_n = m_1$ and $G_{i_0}\not\equiv 0$ for some
$i_0\in I_d$.  We may divide (\ref{equ5en}) by $u_1^{m_1-1}$. As a
consequence, we obtain
\begin{equation}
\label{equ8en} X_1=\dps u_1\sum_{i\in
I_n\setminus I_d}\bigg(Q_i(u)+u_1\widetilde{P_i}(u)\bigg)\frac{\partial}{\partial u_i} +
\sum_{i\in
J_d}\bigg(G_i(u)+u_1\widetilde{P_i}(u)\bigg)\frac{\partial}{\partial
u_i}.
\end{equation}
Unlike the previous case, the leaves of $\fol_1$ restricted to the exceptional divisor are described below 
\begin{equation}
\label{ft3}\bigg\{\varphi(t,p)=(\varphi_1(t),\varphi_2(t),\ldots,\varphi_n(t)), \varphi(0,p)=p=(p_i)_i,t\in D(0,\epsilon\bigg\}\end{equation} such that $\varphi_1(t,0)\equiv0$ and $\varphi_i(t,p_i)\equiv p_i$ are constant functions for $i>d$. More precisely, the leaves of $\fol_1$ in $\E_1$ are tangent to the fibers $\pi_1^{-1}(q)$, $q\in\W_0$.
But in this situation,  the singular set of $\fol_1$  restricted to $\E_1$ is given by the following $d-1$ equations
\begin{equation}\label{sst3}
Q_i(0,u_2,\ldots,u_n)-u_iQ_1(0,u_2,\ldots,u_n)=0,\quad \mathrm{for} \quad i=2,\ldots,d
\end{equation}
which can be generically solved for $d-1$ unknown variables $u_2, \ldots, u_d$ as follows
\begin{equation}
\label{defce2}
\W_i=\{u\in\E_1| u_k=\psi_k(u_{d+1},\ldots,u_n),\quad k=2,\ldots d\}.
\end{equation}
Thus, each singular component $\W_i$ has dimension equal to $d-n$  and is homeomorphic to $\W_0$ with $\pi_1(\W_i)=\W_0$.  Note that $\W_i$ can be given by the graph of the function $\Psi:\W_0\to \E_1$ locally defined as
\begin{equation}\label{for113}
 \Psi(u_{d+1},\ldots,u_n)=(0,\psi_{2}(u_{d+1},\ldots,u_n),\ldots,\psi_d(u_{d+1},\ldots,u_n),u_{d+1},\ldots,u_n).
\end{equation}
since $\W_0$ is locally defined as $z_1=\ldots=z_d=0$  and $u_i=z_i$ for $i>d$.

In the case where $\W_0$ is a dicritical component, it is only
described by the following condition.

\noindent Case (iv) : $m_n = m_1$ and $G_j\equiv0$ for $j\in I_d$. Then, after the division of (\ref{equ5en}) by $u_1^{m_1}$ we get
\begin{equation}
\label{equ9en}  X_1=\dps\sum_{i\in
I_n}\bigg(Q_i(u)+u_1\widetilde{P_i}(u)\bigg)\frac{\partial}{\partial u_i}+
\sum_{i\in J_d}\widetilde{P_i}(u)\frac{\partial}{\partial u_i}.
\end{equation}
In this situation, the exceptional divisor $\E_1$ is not an
invariant set of $\fol_1$.
With the notation given in this section, the number
$\ell:=m_{\E_1}(\pi_1^*\fol_0)$ will be called the order of
annulment of $\pi_1^*\fol_0$ at $\E_1$ is defined as follows.
\begin{equation}
\label{for12} \ell=
\begin{cases}
{\rm min}\lbrace m_1-1, m_n \rbrace, & \mbox{ if } {\W_0 \mbox{ is non-dicritical}} \\
m_1,& \mbox{ if } {\W_0 \mbox{ is dicritical. }}
\end{cases}
\end{equation}
In particular, if $m_1=m_n$ then
\begin{equation}
\label{for18} \ell=
\begin{cases}
\mbox{m}_{\W_0}(\fol_0)-1,& \mbox{ if } {\W_0 \mbox{ is non-dicritical}} \\
\mbox{m}_{\W_0}(\fol_0),& \mbox{ if } {\W_0 \mbox{ is dicritical. }}
\end{cases}
\end{equation}

\begin{definition}\label{defep}
Let $\W_0$ be a non-dicritical component of $\sing(\fol_0)$. We will say that $\W_0$ is of type I, II and III if
$m_n+1=m_1$, $m_n+1< m_1$ and $m_n=m_1$, respectively.
\end{definition}

\begin{definition}
The foliation $\fol_0$ will be called {\it special } along $\W_0$ if
the induced foliation $\fol_1$ has the exceptional divisor
$\E_1=\pi_1^{-1}(W_0)$ as an invariant set, and ${\sing}(\fol_1)$
meets $\E_1$ at isolated singularities at most.
\end{definition}
\begin{remark}\rm  If $\fol_0$ is special along $\W_0$ then necessary $\W_0$ is of type I and  $G_{j}\not\equiv 0$ for all
$j\in I_d$.
\end{remark}
As in \cite{AG}, we will denote by $\nn(\fol_0,Z)$ the singularity number of $\fol_0$ in $Z$, counted with multiplicities,
where $Z$ is a smooth subvariety of $\pn$ of arbitrary
dimension which is an invariant set of $\fol_1$. 
\begin{subsection}{Chern classes} Let us consider a blow-up sequence $\pi_j:\M_{j}\to \M_{j-1}$  along a smooth curve $\W_{j-1}$, with exceptional
divisor $\E_j$ such that $\W_j\subset \E_j$ for all $j\ge1$. Furthermore, we will admit that $\W_j$ is homeomorphic to $\W_{j-1}$ and $\pi_j(\W_j)=\W_{j-1}$. Set $\nn_j:=\nn_{\W_j/\M_j}$ the normal bundle of $\W_j$ in $\M_j$  and $\rho_j:=\pi_j|_{\E_j}$. Since
$\E_j \cong \P(\nn_{j-1})$, recall that $A(\E_j)$ is generated as an
$A(\W_{j-1})$-algebra by the Chern class
$$
\zeta_j := c_1(\oo_{\nn_{j-1}}(-1))
$$
with the single relation
\begin{equation}
\label{equdch} 
\dps\sum_{i=0}^{n-1}(-1)^i\zeta_j^{n-1-i}\cdot\rho_j^*c_i(\nn_{j-1})=0.
\end{equation}
The normal bundle $\nn_{\E_j/\M_j}$ agrees with the tautological bundle $\oo_{\nn_{j-1}}(-1)$, and hence
\begin{equation}
\label{equzet}
\zeta_j=c_1(\nn_{\E_j/\M_j})=[\E_j].
\end{equation}
If $\iota_j :\E_j\hookrightarrow\M_j$ is the inclusion map, we also get
\begin{equation}
\label{equslf}
{\iota_j}_{*}(\zeta_j^i)=(-1)^i [\E_j]^{i+1}.
\end{equation}
Given that
$$
\dps\int_{\E_j}\rho^*c_{i}(\nn_{j-1})\zeta_j^{n-i-1} =
(-1)^{n-i-1}\dps\int_{\W_{j-1}}c_i(\nn_{j-1}) = 0
$$
for $i\ge 2$, we have
\begin{align}
\label{forEE}
\int_{E_j}\zeta_j^{n-1} &= \int_{E_j}\rho^*c_1(\nn_{j-1})\zeta_j^{n-2} = (-1)^{n} \int_{\W_0}c_1(\nn_{j-1}) \\
                    &=(-1)^n\int_{\W_{j-1}}c_1(\tt_{M_{j-1}}\otimes\oo_{\W_{j-1}})-c_1(\tt_{\W_{j-1}}):= (-1)^n\Lambda_j^{(n)} \nonumber
\end{align}

In particular, for $\M_0=\pn$ we have that 
\begin{equation}\label{forlm}
\Lambda_0^{(n)}= (n+1)\deg(\W_0)-\chi(\W_0),  
\end{equation} 
where $\chi(\W_0)$ is the Euler characterist of $\W_0$.
By other side, from Porteous' Theorem (see \cite{IP}), it holds that
\begin{equation}
\label{relcc} c_1(\tt_{\M_j}) =  \pi_j^*c_1(\tt_{\M_{j-1}}) - (n-2)\E_j.
\end{equation}

In particular, from Whitney formula, we have that

\begin{equation}
\label{equpor} 
c_1(\tt_{\W_{j}}) + c_1(\nn_{j})\bigg|_{\W_{j}}=c_1(\tt_{M_{j}})\bigg|_{\W_{j}} = \bigg(\pi_{j}^*c_1(\tt_{\M_{j-1}}) - (n-2)\E_{j}\bigg)\bigg|_{\W_{j}}
\end{equation}
which results for $j=1$ that
\begin{equation}
\label{for20} \int_{\E_2}\zeta_2^{2}=(-1)^n\int_{\W_{1}}c_1(\nn_{1})= (-1)^n\int_{\W_{1}} \bigg(\pi_1^*c_1(\tt_{\M_0}) -c_1(\tt_{\W_{1}})-(n-2)\E_1\bigg).
\end{equation}
Therefore, given that $\W_1$ is homeomorphic to $\W_0$ and $\pi_1(\W_1)=\W_0$ then it is not hard to see that 
\begin{equation}\label{for322}
\dps\int_{\W_{1}}\E_1 = \frac{\Lambda_0^{(n)}}{n-1}
\end{equation}
since $\chi\big(\P^{n-2}\big)=n-1$. In fact,
\begin{equation}\label{for1322}
\dps\int_{\W_{1}}\E_1 = \frac{(-1)^{n-2}}{n-1}\dps\int_{\W_{1}}\E^{n-1}_1=\frac{\Lambda_0^{(n)}}{n-1}
\end{equation}

Thus, the equations (\ref{for20}) and (\ref{for322}) make us to conclude that

\begin{equation}
\label{for21} \int_{\E_2}\zeta_2^{n-1} = (-1)^n\bigg(\Lambda_0^{(n)} -\frac{n-2}{n-1}\Lambda_0^{(n)} \bigg)= \frac{\Lambda_0^{(n)}}{n-1}.
\end{equation}

Now, we will consider a finite induction hypothesis as follows

\begin{equation}
\label{for22} \int_{\E_j}\zeta_j^{n-1} = (-1)^n\frac{\Lambda_0^{(n)}}{(n-1)^{j-1}},\quad  \int_{\W_j}\E_j = \frac{\Lambda_0^{(n)}}{(n-1)^{j}}
\end{equation}
for $j=1,\ldots,k$. So, from (\ref{equpor}), we obtain that
\begin{equation}
\label{for25} 
\begin{array}{lcl}
\dps\int_{\E_{k+1}}\zeta_{k+1}^{n-1}  &=&(-1)^n \dps\int_{\W_{k}}  c_1(\nn_{k}))=(-1)^n\int_{\W_{k}} c_1(\tt_{\M_k}) -c_1(\tt_{\W_{k}})\\
                                        &=&(-1)^n\bigg(\dps\int_{\W_{k}}  \pi_k^*c_1(\tt_{M_{k-1}})-c_1(\tt_{\W_{k}})-(n-2)\E_k\bigg)\\
                    &=&(-1)^n\bigg(\dps\int_{\W_{k-1}}c_1(\tt_{M_{k-1}})-c_1(\tt_{\W_{k-1}}) -(n-2) \int_{\W_{k}}\E_k\bigg)\\
                        &=&(-1)^n\bigg(\dps\int_{\W_{k-1}}  c_1(\nn_{k-1}))-(n-2) \int_{\W_{k}}\E_k\bigg)\\
                        &=&(-1)^n\bigg(\dps\frac{\Lambda_0^{(n)}}{(n-1)^{k-1}}-\frac{n-2}{(n-1)^k}\Lambda_0^{(n)}\big) \\
                        &=&(-1)^n\dps\frac{\Lambda_0^{(n)}}{(n-1)^k}
\end{array}
\end{equation}

From this fact, we get 

\begin{equation}\label{for26}
\int_{\W_{k+1}}\E_{k+1} = \frac{(-1)^{n-2}}{n-1}\int_{\W_{k+1}}\E^{n-1}_{k+1}=\frac{\Lambda_0^{(n)}}{(n-1)^{k+1}}.
\end{equation}
                                                  
\end{subsection}
\begin{theorem}\label{theo2.8} Let $\fol_0$ be a holomorphic foliation by curves on $\M_0=\pn$ such that
$$\sing(\fol_0)=\W_0\cup\{p_1,\ldots,p_r\}$$
where each $p_i$ is a closed point and $\W_0=Z(f_1,\ldots,f_d)$  is a
codimension-$d$ smooth variety with  $k_j=\deg(f_j)$. Let $\pi_1:\M_1\to\M_0$ be
the blow-up of $\pn$ centered along $\W_0$, $\E_1=\pi_1^{-1}(\W_0)$and
$\fol_1$ the induced foliation by $\fol_0$ via $\pi_1$.
\begin{itemize}
\item [(a)] If $\fol_0$ is special along $\W_0$ then

$$\nn(\fol_1,\E_1)=-\nu(\fol_0,\W_0,\psi_a)$$
where
$${\psi}_{a}(x)=\big((1+x)^{d-a_1}-1\big)x^{n-d-a_2-1}.$$

\item [(b)] In addition, if all the singularities of $\fol_1$ are isolated closed points then
$$\nn(\fol_1,\M_1)=\sum_{i=0}^{n}k^i+\nu(\fol_0,\W_0,\varphi_a-\psi_a).$$
\end{itemize}
\end{theorem}
\begin{proof}
To proof the item (a) see the proof of theorem 3.2 in \cite{AG}. More precisely, see equations (24), (37), (40) and (41) in \cite{AG}. To proof the item (b), it is enough to observe the equation (37) in \cite{AG} which is written below

$$
\nn(\fol_1,\M_1)=\dps\sum_{i=0}^{n}k^i + \deg(\W_0)\dps\sum_{|a|=0}^{n-d}\sum_{j=|a|}^{n}\sum_{m=0}^{n-d-|a|}(-1)^{\delta_{|a|}^{m}}{{n-j}\choose{m}}\Gamma_{a}^{j}\ell^{n-j-m}(k- 1)^m\sigma_{a_1}^{(d)}\tau_{a_2}^{(d)}\mathcal{W}_{\delta_{|a|}^{m}}^{(d)}
$$
where 
\begin{equation}
\label{gamma}
\Gamma^{j}_{a}={{d-a_1}\choose{j-|a|-1}}-{{d-a_1}\choose{j-|a|}}.
\end{equation}
and $\ell$ given by (\ref{for12}). Here we are assuming that 
$\dps{{p}\choose{q}}=0$ if $p<q$ or $q<0$.
However,
$$ \sum_{j=|a|+1}^{n}{{n-j}\choose{m}}{{d-a_1}\choose{j-|a|}-1}\ell^{n-j-m}=\frac{{\psi}_{a}^{(m)}(\ell)}{m!},$$
$$\dps\sum_{j=|a|}^{n}{{n-j}\choose{m}}{{d-a_1}\choose{j-|a|}}\ell^{n-j-m}=\frac{\varphi_{a}^{(m)}(\ell)}{m!}$$
which yields
$$\nn(\fol_1,\M_1)=\sum_{i=0}^{n}k^i+\nu(\fol_0,\W_0,\varphi_a-\psi_a).$$
\end{proof}
\section{Fundamental Lemma of deformation}

In this section, we will present a generalization for the
Fundamental Lemma of deformation given in \cite{AG}. Essentially, we
will treat the case in which the singular set of $\fol_0$ contains a
dicritical component.

\begin{lema}\label{lema4}
Let $\fol_0$ be a one-dimensional holomorphic foliation on $\M_0=\pn$,
$n\ge 3$, of degree $k$. Suppose that
$$
 \mbox{Sing}(\fol_0) = \W_0 \cup\{p_1,\ldots,p_s\},
$$
where $\W_0=Z(f_{1},\ldots,f_{d})$ is a smooth complete
intersection of $\pn$ and $p_j$ are isolated points disjoint to
$\W_0$. As before, $\pi_1:\M_1\to\M_0$ is the blow-up of $\pn$  along  $\W_0$, with an exceptional divisor $\E_1$. Then, there exists a one-parameter family of one-dimensional holomorphic foliations on
$\pn$, denoted by $\{\fol_t\}_{t\in D}$ where
$D=D(0,\epsilon)=\{t\in\C\,:\,|t|<\epsilon\}$ such that
\begin{enumerate}
\item[(i)] $\dps\lim_{t\to0}\fol_t=\fol_0$;
\item[(ii)] $\deg(\fol_t)=\deg(\fol)$ for all $t\in D$;
\item[(iii)] If $\mbox{m}_{\W_0}(\fol_0)=1$ then $\W_0$ is an invariant set of $\fol_t$ for any $t\in D\setminus\{0\}$.
\item[(iv)] If $\mbox{m}_{\W_0}(\fol_0)\ge2$ then $\sing(\fol_t) = \W_0 \cup\{p_1^t,\ldots,p_{s_t}^t\}$ and $\fol_t$ is special along $\W_0$ for any $t\in D\setminus\{0\}$;
\item[(v)] For any $t\in D\setminus\{0\}$, the order of annulment of $\pi_1^{*}(\fol_t)$ at $\E_1$ is
$$ m_{\E_1}(\pi_1^*\fol_t)=\begin{cases}
m_{\E_1}(\pi_1^*\fol_0)& \mbox{ if } {\W_0 \mbox{ is non-dicritical}} \\
m_{\E_1}(\pi_1^*\fol_0)-1& \mbox{ if } {\W_0 \mbox{ is dicritical.
}}
                         \end{cases}$$
\end{enumerate}
\end{lema}

\begin{proof}
Consider that foliation $\fol_0$ is described by the
polynomial vector field
$X_0=\dps\sum_{i=1}^{n}P_i(z)\frac{\partial}{\partial z_i}$ in some
open affine set $U_j\subset \pn$. Given that $\W_0$ is smooth
variety, by a reordering of the variables, if necessary, we can
admit that the $d\times d$ matrix
$$\mathbf{M} = \dps\bigg[\frac{\partial f_i}{\partial z_j}\bigg]_{1\le i,j\le d}$$
is not singular in some open set $U$ such that $\W_0\cap
U\neq\emptyset$. Therefore, $F:U\subseteq U_j\to V\subseteq \C^n$
defined as

$$w=F(z)=(f_1(z),\ldots,f_d(z),z_{d+1},\ldots,z_n)$$
is a local biholomorphism. Furthermore, the image $F(\W_0\cap
U)=\widetilde{\W_0}$ is defined as $w_1=\ldots=w_d=0$. Let
$\G_0=F_*(\fol_0)$ be the push-forward foliation defined in $V$ which
is described by the vector field
$$
X_{\G_0}=Q_1(w)\,\frac{\partial}{\partial w_1}+\ldots
+Q_n(w)\,\frac{\partial}{\partial w_n}
$$
where
\begin{equation}
\label{equgil} Q_i(w)=\sum_{|a|=m_i}w_1^{a_1}\cdots
w_{d}^{a_{d}}Q_{i,a}(w)
\end{equation}
with at least one $Q_{i,a}(z)$ not vanishing at $\widetilde{\W_0}$. Thus,
\begin{equation}
 Q_{i}\circ F(z)=\left\{\begin{array}{ll}
                                    \dps\sum_{j=1}^{n}\frac{\partial f_i}{\partial z_j}\cdot P_j(z),&i=1,\ldots,d\cr
                                    P_i(z),&i=d+1,\ldots,n
                                  \end{array}\right.
\end{equation}

Solving this system and applying the factor $\det(\mathbf{M})$ for
normalizing, we can admit that
\begin{equation}\label{formPi}
P_{i}(z)=\left\{\begin{array}{ll}
                                    \det(A_i(z)),&i=1,\ldots,d\cr
                                    \det(\mathbf{M})\cdot Q_{i}\circ F(z),&i=d+1,\ldots,n
                   \end{array}\right.
\end{equation}
where $A_i$ is obtained replacing the $i$-th column of $DF$ by the
vector column $(Q_1\circ F(z),\ldots,Q_n\circ F(z))$.

We consider the one-dimensional holomorphic foliation $\G_{t}$
defined in $V$ described by the following vector field
\begin{equation}\label{deformation}
X_{\G_{t}}=X_{\G_0}+t\cdot Y(w)
\end{equation}
where
$$Y(w)=Y_1(w)\,\frac{\partial}{\partial w_1}+\ldots +Y_n(w)\,\frac{\partial}{\partial w_n}
$$
with
$$Y_i(w)=\sum_{|a|=q_i}w_1^{a_1}\cdots
w_{d}^{a_{d}}Y_{i,a}(w),\quad q_i=\mbox{m}_{\widetilde{\W_0}}(Y_i)$$
with at least one $Y_{i,a}(z)$ not vanishing at $F(\W_0\cap U)$ for
all $i$ such that
\begin{equation}
\label{multp}
q_1=q_2=\ldots=q_d=q_{d+1}+1=\ldots=q_{n}+1=1+\ell.
\end{equation}
where 
$$\ell = \left\{\begin{array}{ll}
                      \mbox{m}_{\E_1}(\pi_1^*\fol_0),&\mbox{ if }\W_0 \mbox{ non-dicritical}\cr
                      \mbox{m}_{\E_1}(\pi_1^*\fol_0)-1,&\mbox{ if }\W_0 \mbox{ dicritical}
                      \end{array} \right.$$

The one-parameter family of one-dimensional holomorphic foliations
$\fol_t$ is defined as pull-back of $\G_t$, $\fol_t=F^*\G_t$.
Therefore, $\fol_t$ is described by the vector field $X_t$ as
follows
$$ X_t = \dps\sum_{i=1}^{n}P_i^t(z)\frac{\partial}{\partial z_i}$$
where $P_i^t(z)$ is obtained from (\ref{formPi}) chancing $Q_i\circ
F(z))$ by $Q_i\circ F(z))+tY_i\circ F(z)$. It is immediate that $\dps\lim_{t\to0}\fol_t=\fol_0$. The vector
field $X_t$ is polynomial since each $P_i$ and $F$ they are also.
Then, we can consider $\fol_t$ defined in the open affine set $U_j$
by Hartogs Extension Theorem. The proof of (i) is immediate. The
functions $Y_{i,a}$ are chosen so that $\deg(\fol_t)$ be equal to
$\deg(\fol_0)$. And for that, some functions $Y_i$ may be constant,
null or an affine linear in variables $w_{d+1},\ldots,w_n$, proving
(ii). If $\mbox{m}_{\W_0}(\fol_0)=1$ then in (\ref{multp}) $q_i=0$
for $i>d$ which results that $\W_0$ is an invariant set of $\fol_t$
for $t\neq0$. Otherwise, if $\mbox{m}_{\W_0}(\fol_0)\ge2$, $\W_0
\subset\sing(\fol_t)$ for all $t\in D$. Furthermore, shrinking
$\epsilon$, if necessary, we can admit that $\sing(\fol_t)$ contains
isolated closed points disjoint from $\W_0$, since $F$ is a local biholomorphism. By construction, for $\fol_t$, with $t\ne0$,  the curve $\W_0$ is of type I, which means that by changing some coefficients of
$Y$ we can admit that $\fol_t$ is special along $\W_0$ for $t\in
D\setminus\{0\}$, proving (iv). Finally, for $t\neq0$,
$\mbox{m}_{\E_1}(\pi_1^*\fol_t)=q_n=\mbox{mult}_{\W_0}(\fol_0)-1$.
Thus, from (\ref{for18}) we get (v).
\end{proof}
\subsection{Embedded closed points}
Lemma \ref{lema4} allows us to determine the embedded closed points associated with $\W_0$ more effectively. In fact, let $\fol_t$ be a {\it special} deformation of $\fol_0$ given by the Lemma \ref{lema4} for $t\in D(0,\epsilon)$. Thus,  or $\W_0\subset\sing(\fol_t)$ or $\W_0$ is $\fol_t$-invariant for all $t\ne0$. Whatever,  we can assume that all the isolated singularities of $\fol_t$ are non-degenerate, so we set
\begin{equation}
A_{\W_0}:=\{p^{t}_{j}\in\sing(\fol_t): \displaystyle\lim_{t\to
0}p^{t}_{j}\in \W_0,\quad\text{such that}\quad p^{t}_{j}\not\in\W_0,
\forall t\ne0\} .
\end{equation}
where each $p^{t}_{j}$ is isolated point.  We will
indicate by $ N (\fol_0, A_{\W_0}) $ the number of elements $p^{t}_
{j}\in A_{\W_0} $, counted with multiplicities, and such points are called embedding points
associated to $\W_0$.

\subsection{Proof of Theorem \ref{theorem1}}

Let $\fol_0$ be a foliation by curves defined on $\pn$ of degree $k$
such that
$$
 \mbox{Sing}(\fol_0) =\W_0 \cup\{p_1,\ldots,p_s\},
$$
which $\W_0=Z(f_1,\ldots,f_m)$.  By Lemma \ref{lema4}, there is a
one-parameter family of holomorphic foliations by curves on $\pn$, given by $\{\fol_t\}_{t\in D}$ where $D=D(0,\epsilon)$
satisfies $(i)-(v)$ conditions. Thus, for
$\mbox{m}_{\W_0}(\fol_0)\ge2$, we have
$$\sing(\fol_t)=\W_0\cup\{p_1^t,\ldots,p_{s_t}^t\}$$
where each $p_i^t$ is a closed point. Let $\pi_1:\M_1\to\M_0$ be the blowup of $\M_0=\pn$ along  $\W_0$ being $\E_1$ and
$\widetilde{\fol_t}$ the exceptional divisor and the induced
foliation in $\M_1$, respectively. Thus,
$$\sing(\widetilde{\fol_t})=\{\widetilde{p_1^t},\ldots,\widetilde{p_{r_t}^t}\}.$$
Since $\pi_1:\M_1\setminus\E_1\to\M_0\setminus\W_0$ is a
biholomorphism, we have that

$$\dps\sum_{j=1}^{s}\mu(\fol_0,p_j)=\lim_{t\to0}\sum_{\dps\lim_{t\to0}\widetilde{p_j^t} \notin
\E_1}\mu(\widetilde{\fol_t},\widetilde{p}_j^t).$$  By other hand,

$$\dps\sum_{j=1}^{s}\mu(\fol_0,p_j)=\nn(\widetilde{\fol_t},\M_1) - \lim_{t\to0}\sum_{\dps\lim_{t\to0}\widetilde{p}_j^t \in
\E_1}\mu(\widetilde{\fol_t},\widetilde{p}_j^t).$$

\par Given that $\fol_t$ is special along $W_0$ for $t\ne0$, from Theorem \ref{theo2.8}, we get

$$\begin{array}{lcl}
\dps\sum_{j=1}^{s}\mu(\fol_0,p_j)&=&\nn(\widetilde{\fol_t},\M_1) -
\nn(\widetilde{\fol_t},\E_1) - N(\fol,A_{\W_0})\cr
 &=&\dps\sum_{i=1}^{n}k^i+\nu(\fol_0,\W_0,\varphi_a-\psi_a)+\nu(\fol_0,\W_0,\psi_a)-N(\fol,A_{\W_0}).
\end{array}$$ Then,

$$\dps\sum_{j=1}^{s}\mu(\fol_0,p_j)=\dps\sum_{i=1}^{n}k^i +\nu(\fol_0,\W_0,\varphi_a) - N(\fol,A_{\W_0}).$$

If $\ell=\mbox{m}_{\E_1}(\pi_1^*\fol_0)=0$, then $\W_0$ is
$\fol_t$-invariant, resulting in
$$\dps\sum_{j=1}^{s}\mu(\fol_0,p_j)=\lim_{t\to0}\sum_{\dps\lim_{t\to0} p_j^t \notin
\W_0}\mu(\widetilde{\fol_t},p_j^t).$$
Therefore,
$$\dps\sum_{j=1}^{s}\mu(\fol_0,p_j)=\nn(\fol_t,\pn) - \lim_{t\to0}\sum_{\dps\lim_{t\to0} p_j^t \in
\W_0}\mu(\fol_t,p_j^t).$$
Thus,
$$\dps\sum_{j=1}^{s}\mu(\fol_0,p_j)=\nn(\fol_t,\pn) - \nn(\fol_t,\W_0) - N(\fol,A_{\W_0}).$$

By \cite{AG}, we get
$$\dps\sum_{j=1}^{s}\mu(\fol_0,p_j)=\dps\sum_{i=1}^{n}k^i
+\nu(\fol_0,\W_0,\varphi_a)|_{\ell=0} - N(\fol,A_{\W_0}).$$
It concludes the prove of Items (a) and (b) of Theorem \ref{theorem1}.
Then,
$$\dps\mu(\fol_0,\W_0)=-\nu(\fol_0,\W_0,\varphi_a) + N(\fol_0,A_{\W_0}).$$

Note that  $N(\fol_0,A_{\W_0})$ is a finite number since $\fol_0$ is represented by a polynomial vector field and $$N(\fol_0,A_{\W_0})\le \nn(\widetilde{\fol}_t,\M_1)=\sum_{i=0}^{n}k^i+\nu(\fol_0,\W_0,\varphi_a-\psi_a).$$
By the same way,
$$\dps\mu(\fol_1,\bigcup_i\W_i^{(1)})=\nn(\F_1,\E_1)+N(\fol_1,A_{\E_1})=-\nu(\fol_0,\W_0,\psi_a)+N(\fol_1,A_{\E_1})$$
where each $\W_i^{(1)}$ is a connected component of $\sing(\fol_1)$ contained in $\E_1$.
But, since
$\pi_1|_{\M_1\setminus\E_1}:\M_1\setminus\E_1\to\M_0\setminus\W_0$ is
a biholomorphism, we get $N(\fol_0,A_{\W_0})=N(\fol_1,A_{\E_1})$,
i.e.,

$$\mu(\fol_1,\bigcup_i\W_i^{(1)})=\dps\mu(\fol_0,\W_0)+\nu(\fol_0,\W_0,\varphi_a)-\nu(\fol_0,\W_0,\psi_a)$$
But,
$$\nu(\fol_0,\W_0,\varphi_a)-\nu(\fol_0,\W_0,\psi_a)=\nu(\fol_0,\W_0,\varphi_a-\psi_a)=\nu(\fol_0,\W_0,\vartheta_a)$$
which concluded the proof of Item (c) of Theorem \ref{theorem1}.
\begin{remark} \rm In theorem \ref{theorem1}, we have in mind that  if $\W_i^{(1)}\cap \W_j^{(1)}\neq\emptyset$ then $\mu(\fol_1,\W_i^{(1)}\cup \W_j^{(1)})=\mu(\fol_1,\W_i^{(1)})+\mu(\fol_1,\W_j^{(1)})-N(\fol_1,\W_i^{(1)}\cap \W_j^{(1)}).$
\end{remark}

\begin{remark} \rm If $\W_0=Z(f_1,\ldots,f_n)$ with $\mbox{deg}(f_j)=1$
for $j=1,\ldots,n$ then $\W_0$ is an isolated closed point. Theorem \ref{theorem1} assures that
$$\mu(\fol_1,\bigcup_i\W_i^{(1)})=\mu(\fol_0,\W_0)+ \vartheta_{0}(\ell)$$
where
$$\vartheta_{0}(\ell)=(1+\ell)^n+\dps\frac{1-(1+\ell)^n}{\ell}=(1+\ell)^n-\sum_{j=0}^{n-1}(1+\ell)^j.$$
Thus, if $\W_0$ is a non-dicritical component then
$\ell=\mbox{m}_{\W_0}(\fol_0)-1$ which results
$$\mu(\fol_1,\bigcup_i\W_i^{(1)})=\mu(\fol_0,\W_0)+(\mbox{m}_{\W_0}(\fol_0))^n-\sum_{j=0}^{n-1}(\mbox{m}_{\W_0}(\fol_0))^j.$$

On the other hand, if $\W_0$ is a dicritical component, then
$\ell=\mbox{m}_{\W_0}(\fol_0)$ resulting in the following result.

$$\mu(\fol_1,\bigcup_i\W_i^{(1)})=\mu(\fol_0,\W_0)+(1+\mbox{m}_{\W_0}(\fol_0))^n-\sum_{j=0}^{n-1}(1+\mbox{m}_{\W_0}(\fol_0))^j.$$
These results agree with \cite{rb} for the case which $\W_0$ is
absolutely isolated singularity. However, these expressions remain true
even if the singular set of $\fol_1$ contains components of positive dimension.
\end{remark}
\begin{remark} \rm  If $\W_0=Z(f_1,\ldots,f_{n-1})$ with
$\mbox{deg}(f_j)=k_j$ then $\W_0$ is a smooth curve. Thus,
$\tau^{(n-1)}_1=\big(n+1-\sum_{j=1}^{n-1}{k_j}\big)$ and
$\sigma_1^{(n-1)}=\sum_{j=1}^{n-1}{k_j}$, resulting in
$\chi(\W_0)=\tau^{(n-1)}_1\cdot\mbox{deg}(\W_0)$. By Theorem
\ref{theorem1}, we get
$$\mu(\fol_1,\bigcup_i\W_i^{(1)})=\mu(\fol_0,\W_0)+\chi(\W_0)\bigg(\sum_{j=0}^{n-3}(1+\ell)^j-\ell^2(1+\ell)^{n-2}\bigg)+$$
$$(1+\ell)^{n-2}\mbox{deg}(\W_0)\bigg((n-n\ell-2)(k-1)+(n+1)(\ell^2-\ell)\bigg),$$
where
$$\ell=\left\{\begin{array}{ll}
              m_{\E_1}\big(\pi_1^*\fol_0\big),& \mbox{ if } {\W_0 \mbox{ is non-dicritical}}\cr
              m_{\E_1}\big(\pi_1^*\fol_{0}\big)-1,&\mbox{ if } {\W_0 \mbox{ is dicritical}}
              \end{array}\right. $$
\end{remark}
\begin{example} \rm  Let us consider the one-dimensional holomorphic
foliation $\fol_0$ of degree $3$ defined on $\M_0=\P^3$ described in the
open affine set $U_3=\{[\xi_i]\in\P_3|\xi_3\neq0\}$ by the vector field
$$X=\big(P_2(z)+P_3(z)\big)\frac{\partial}{\partial z_1}+\big(Q_2(z)+Q_3(z)\big)\frac{\partial}{\partial
z_2}+R_2(z)\frac{\partial}{\partial z_3}$$ where
$P_i(z)=\sum_{j=0}^{i}p_{ij}z_1^{i-j}z_2^j$,
$Q_i(z)=\sum_{j=0}^{i}q_{ij}z_1^{i-j}z_2^j$,
$R_2(z)=\sum_{j=0}^{2}r_{j}(z)z_1^{2-j}z_2^j$ with
$p_{ij},q_{ij}\in\C$,
$r_{j}(z)=\alpha_jz_1+\beta_jz_2+\gamma_jz_3+\delta_j$
$z_i=\xi_{i-1}/\xi_3$.  We will also assume
$z_1Q_2(z)-z_2P_2(z)\equiv0$ which results that
$\W_0=\{[\xi]\in\P_3|\xi_0=\xi_1=0\}$ is a dicritical component of
$\sing(\fol_0)$. Thus, $P_2(z)=p_{20}z^2+p_{11}z_1z_2$ and
$Q_2(z)=p_{20}z_1z_2+p_{11}z_2^2$. Let $p_i(\lambda)=P_i(1,\lambda)$
and $q_i(\lambda)=Q_i(1,\lambda)$ for $i=2,3$. We admit that both $p_2$ and $p_3$ have no roots in common. Beyond
this curve, the singular set of $\fol_0$ has 8 more isolated points,
counting the multiplicities. Let $\pi_1:\M_1\to\M_0$, $\E_1$ and
$\fol_1$ as before. It is not difficult to see that
$m_{\E_1}(\pi_1^*\fol_0)=2$ and $\sing(\fol_1)$ has 12 isolated
points, counting the multiplicities, 4 of them in $\E_1$.

Let $\fol_t$ be the one-parameter family of holomorphic foliation on
$\P^3$ described by the vector field
$$X_t=X+t\bigg(z_3A_2(z)\frac{\partial}{\partial z_1}+z_3B_2(z)\frac{\partial}{\partial
z_2}+C_1(z)\frac{\partial}{\partial z_3}\bigg)$$ where
$A_2(z)=a_0z_1^2+a_1z_1z_2+a_2z_2^2$,
$B_2(z)=b_0z_1^2+b_1z_1z_2+b_2z_2^2$ and
$C_1(z)=z_1c_0(z)+z_2c_1(z)$ being each $c_i$ an affine linear
function. Varying the coefficients of $A_2,B_2$ and $C_1$ if necessary, we can admit that $\fol_t$ is special along $\W_0$, $\mbox{m}_{\E_1}(\pi_1^*\fol_t)=1$ and
$\mbox{deg}(\fol_t)$=3 for each $t\neq0$. Let $\folt_t$ be the induced foliation on $\M_1$
by $\fol_t$ via $\pi_1$. Therefore, for $t\ne0$ the singular set of $\fol_t$  has 20
isolated closed points, counting the multiplicities since $\nn(\folt_t,\E_1)=10$ and $\nn(\folt_t,\M_1)=30$. See Theorem \ref{theo2.8}.  As consequence, $N(\fol_0,A_{\W_0} )=12$  because
the foliation $\fol_1$ has 8 isolated point outside the exceptional divisor. Therefore, we get

$$\mu(\fol_0,\W_0)=-\dps\lim_{t\to0}\nu(\fol_t,\W_0,\varphi_a)+N(\fol_0,A_{\W_0})=20+ 12 =32.$$

This result is totally coherent and compatible, since a generic
perturbation of $\fol_0$ with the same degree of $\fol_0$ will
produce 40 isolated singularities, which means that the curve $\W_0$
corresponds to 32 isolated singularities. However, if $\W_0$ is a
dicritical component of $\sing(\fol_0)$ with
$\mbox{m}_{\E_1}(\pi_1^*\fol_0)=2$ and $\deg(\fol_0)=3$, then
$\mu(\fol_0,\W_0)\ge20$.

\end{example}

\section{Holomorphic foliations on $\P^n$}

In this section, we will only consider
a one-dimensional holomorphic foliation $\fol_0$ of degree $k$ defined
in $\M_0=\P^n$, $n\ge3$, such that its singular locus $\sing(\fol_0)$ contains a smooth curve $\W_0$ of degree $\deg(\W_0)$ and the Euler
characteristic $\chi(\W_0)$.

For each point $p\in\W_0$ there is an open set $U_0\subset\M_0$ and $n-1$ polynomials $f_1,\ldots,f_{n-1}$ such that $p\in U_0$ and $\W_0\cap U_0$ are defined as $f_1=\ldots=f_{n-1}=0$ since $\W_0$ is also a smooth curve, which implies that $\W_0$ is a complete local intersection (lci). Shrinking an open set $U_0$ or reordering the variables, if necessary,  we can admit that $\varphi_0: U_0\to V_0\subset\C^n$ defined as $$z=\varphi_0(w)=(f_1(w),\ldots,f_{n-1}(w),w_n)$$ is a local  biholomorphism. Therefore, in $V_0$ the curve $\W_0$ is 
is given by $z_1=\ldots=z_{n-1}=0$ and $\fol_0$ is described by the following vector field  
\begin{equation}\label{for114}
X_0=P_1(z)\frac{\partial}{\partial z_1}+P_2(z)\frac{\partial}{\partial z_2}+\ldots+P_n(z)\frac{\partial}{\partial z_n}
\end{equation}
where each
$$P_i(z)=\sum_{j=m_i}P_{ij}(z),\quad P_{ij}(z)=\sum_{|a|=j}z_1^{a_1}\cdots z_{n-1}^{a_{n-1}}p_{i,j,a}(z_n),\quad m_i=\mbox{m}_{\W_0}(P_i)$$
with $a=(a_1,\ldots,a_{n-1})\in \Z^{n-1}$ and $m_n\le m_{n-1}=\ldots=m_2=m_1$. At first, we assume that there is a blow-up sequence $\{\pi_i, \M_i,\W_i,\fol_i,\E_i\}$ where $\W_i\subset\E_i$ is a homeomorphic curve to $\W_{i-1}$ and $\pi_i(\W_i)=\W_{i-1}$ for all $i\ge1$.
In the chart $\big(U_1,\sigma_1(u)\big)$, the pull-back foliation $\pi_1^*\fol_0$ is described by the following vector field
\begin{equation}\label{for120}
X_1 =u_1^{m_1}Q_1(u)\frac{\partial}{\partial u_1}+ u_1^{m_1-1}\sum_{i=2}^{n-1}\big(Q_i(u)-u_iQ_1(u)\big)\frac{\partial}{\partial u_i}+ u_1^{m_n}Q_n(u)\frac{\partial}{\partial u_n}
\end{equation}
where 
$$ Q_i(u)=\sum_{j=m_i}u_1^{j-m_i}Q_{ij}(u),\quad Q_{ij}(u)=P_{ij}(1,u_2,\ldots,u_n).$$
By hypothesis, there is a curve $\W_1\subset\sing(\fol_1)$ which is locally defined
as $$u_1=u_2-\psi_2(u_n)=\ldots=u_{n-1}-\psi_{n-1}(u_n)=0$$ for certain functions $\psi_i$, for $i=2,\ldots,n-1$.  To continue with this blow-up process, we need to rewrite each $Q_{ij}$ as follows
    $$Q_{ij}(u)=\sum_{|a|=n_{ij}}\big(u_2 -\psi_2(u_n)\big)^{a_1}\cdots\big(u_{n-1} -\psi_{n-1}(u_n)\big)^{a_{n-1}}\widetilde{h}_{i,j,a}(u),\quad n_{ij}=\mbox{m}_{\W_1}(Q_{ij})$$
where $0\le n_{ij}\le j$ and or $\widetilde{h}_{i,j,a}|_{\W_1}\not\equiv0$ or $\widetilde{h}_{i,j,a}\equiv0$ .

If $\W_0$ is of type I then  in $(U_1,\sigma_1(u))$,  the foliation $\fol_1$ is described  by the following vector field 
\begin{equation}\label{for121}
X_1 =u_1Q_1(u)\frac{\partial}{\partial u_1}+ \sum_{i=2}^{n-1}\big(Q_i(u)-u_iQ_1(u)\big)\frac{\partial}{\partial u_i}+ Q_n(u)\frac{\partial}{\partial u_n}
\end{equation}

In this situation, the singular set of $\fol_1$ restricted to $\E_1$ is defined by the equations
$$ u_1=Q_{2,m_1}(u)-u_2Q_{1,m_1}(u)=\ldots=Q_{n-1,m_1}(u)-u_{n-1}Q_{1,m_1}(u)=Q_{n,m_n}(u)=0$$
which may be composed of curves and points. But,  in the coordinate system 
\begin{equation}\label{nscoord}
 v=F(u)=(u_1,u_2-\psi_2(u_n),\ldots,u_{n-1}-\psi_{n-1}(u_n), u_n)
 \end{equation}
the vector field (\ref{for121}) is rewritten  as follows    
\begin{equation}\label{for122}
X_1 =v_1R_1\frac{\partial}{\partial v_1}+ \sum_{i=2}^{n-1}\bigg(R_i-\big(v_i+\psi_i(v_n)\big)R_1-\psi_i^{\prime}(v_n)R_n\bigg)\frac{\partial}{\partial v_i}+ R_n\frac{\partial}{\partial v_n}
\end{equation}
where 
$$ R_i(v)=\sum_{j=m_i}^{\infty}v_1^{j-m_i}R_{ij}(v),\quad R_{ij}(v)=\sum_{|a|=n_{ij}}v_1^{a_1}\cdots v_{n-1}^{a_{n-1}}h_{i,j,a}(v)$$
where $h_{i,j,a}=\widetilde{h}_{i,j,a}\circ F^{-1}$.

Now, if $\W_0$ is of type II then in the chart $\big((U_1)_1, \sigma_1(u)\big)$ the foliation $\fol_1$ is described by the following
vector field as in (\ref{equ7en}), i.e.,

\begin{equation}
\label{exemp2} X_1=
u_1^{m_1-m_n}Q_1(u)\frac{\partial}{\partial
u_1} +
u_1^{m_1-m_n-1}\sum_{i=2}^{n-1}\big(Q_i(u)-u_iQ_1(u)\big)\frac{\partial}{\partial
u_i}\
+Q_n(u)\frac{\partial}{\partial
u_n}.
\end{equation}
On the exceptional divisor $\E_1$,  the leaves of $\fol_1$ are generically defined as 
$u_1=u_2-\alpha_2=\ldots=u_{n-1}-\alpha_{n-1}=0$, where each $\alpha_i$ is a scalar, while its
singular set is defined as $u_1=Q_{n,m_n}(u)=0$. Thus, in this coordinate system (\ref{nscoord}), the vector field  in (\ref{exemp2}) is given by

\begin{equation}\label{for123}
X_1 =v_1^{m_1-m_n}R_1\frac{\partial}{\partial v_1}+v_1^{m_1-m_n-1}\bigg(\sum_{i=2}^{n-1}\big(R_i-\big(v_i+\psi_i(v_n)\big)R_1\big)-\psi_i^{\prime}(v_n)R_n\bigg)\frac{\partial}{\partial v_i}+ R_n\frac{\partial}{\partial v_n}
\end{equation}
with $R_i$ as in (\ref{for122}).
By other side, for generically fixed $u_2,\ldots,u_{n-1}$, there are singular points defined as $u_1=0$ and $u_n=\psi_n(u_2,\ldots,u_{n-1})$
for some function $\psi_n$. Therefore, by continuity, in the singular set of $\fol_1$  there is at least one variety homeomorphic to $\P^{n-2}$. 

If $\W_0$ is of type III then the foliation $\fol_1$ is
described by the following vector field as in (\ref{equ8en}),
\begin{equation}\label{examp3}
X_1=u_1Q_1(u)\frac{\partial}{\partial
u_1} +
\sum_{i=2}^{n-1}\bigg(Q_i(u)-u_iQ_1(u)\bigg)\frac{\partial}{\partial u_i}+u_1Q_n(u)\frac{\partial}{\partial
u_n}
\end{equation}

In this situation, the leaves of $\mathcal{F}_1$ on $\E_1$ are contained in $u_1=u_n-\alpha_n=0$, where $\alpha_n$ is a scalar, while its singular set is defined by $$u_1=Q_{2,m_1}(u)-u_2Q_{1,m_1}(u)=\ldots=Q_{n-1,m_1}(u)-u_{n-1}Q_{1,m_1}(u)=0.$$ Moreover, when $u_n$ is generically fixed, there are singular points which are defined by $$u_1=u_2-\psi_2(u_n)=\ldots=u_{n-1}-\psi_{n-1}(u_n)=0,$$ for some functions $\psi_i$. Consequently, within the singular set of $\mathcal{F}_1$, there exist curves that are homeomorphic to $\W_0$. In the coordinate system (\ref{nscoord}), the vector field in Equation (\ref{examp3}) can be expressed as follows:
\begin{equation}\label{for124}
X_1 =v_1R_1\frac{\partial}{\partial v_1}+ \sum_{i=2}^{n-1}\bigg(R_i-\big(v_i+\psi_i(v_n)\big)R_1-v_1\psi_i^{\prime}(v_n)R_n\bigg)\frac{\partial}{\partial v_i}+ v_1R_n\frac{\partial}{\partial v_n}
\end{equation}
where $ R_i$ is as defined in Equation (\ref{for122}).
 However, unlike the previous case, a variety homeomorphic to $\P^{n-2}$ may not appear unless $$Q_i(0,u_2,\ldots,u_{n-1},\alpha_n)-u_iQ_1(0,u_2,\ldots,u_{n-1},\alpha_n)\equiv0$$ for some $i=2,\ldots,n-1$, 
 for some scalar $\alpha_n$.

Finally, if $\W_0$ is a dicritical component of $\sing(\fol_0)$ then in the chart $(U_1,\sigma_1(u))$, the foliation $\fol_1$ is described by the following vector field
\begin{equation}\label{examp4}
X_1=Q_1(u)\frac{\partial}{\partial
u_1} +
\sum_{i=2}^{n-1}G_i(u)\frac{\partial}{\partial u_i}+Q_n(u)\frac{\partial}{\partial
u_n}
\end{equation}
where 
$$G_i=\dps\sum_{j=m_1+1}^{\infty}u_1^{j-m_1-1}\bigg(Q_{ij}(u)-u_iQ_{1j}(u)\bigg).$$
Thus, $\E_1$ is no longer an invariant set by $\fol_1$ and the singular set $\sing(\fol_1)$ restricted to $\E_1$ is given by equations $$u_1=Q_{1,m_1}(u)=Q_{i,m_1+1}(u)-u_iQ_{1,m_1+1}(u)=Q_{n,m_n}(u)=0$$
for $i=2,\ldots,n-1$. Therefore, in the coordinate system (\ref{nscoord}), the vector field (\ref{examp4}) can be rewritten as follows
\begin{equation}\label{for800}
X_1 =R_1\frac{\partial}{\partial v_1}+ \sum_{i=2}^{n-1}\bigg(\sum_{j=m_1+1}v_1^{j-m_1-1}G_{ij}(v) -\psi_1^{\prime}(v_n)R_n\bigg)\frac{\partial}{\partial v_i}+ R_n\frac{\partial}{\partial v_n}
\end{equation}
where $G_{ij}(v)=R_{ij}(v)-(v_i+\psi_i(v_n))R_{1j}(v)$, $R_i$ and $R_{ij}$ are defined as in (\ref{for122}).

\begin{proposition}\label{propl} Let $\{\pi_j,\M_j,\W_j,\fol_j,\E_j\}$ be a blow-up sequence such that $\W_j$ is homeomorphic to $\W_{j-1}$ with $\pi_j(\W_j)=\W_{j-1}$, where $\W_0$ is a smooth curve and $\M_0=\P^n$. Then

$$m_{\W_j}(\fol_j)\le 1 +m_{\W_0}(\fol_0),\quad \forall j\ge 1.$$
\end{proposition}
\begin{proof}
In fact, let us consider the foliation $\fol_i$ is described by the following vector field

$$ X_i = P_1^{(i)}(v)\frac{\partial }{\partial v_1}+P_2^{(i)}(v)\frac{\partial }{\partial v_2}+\ldots+P_n^{(i)}(v)\frac{\partial }{\partial v_n}$$
where $X_0$ is as in (\ref{for114}) and $X_1$ as in (\ref{for122}), (\ref{for123}),(\ref{for124}) or (\ref{for800}), depending on the type of $\W_1$ is. But, whatever the case, $P_n^{(1)}(v)=R_n^{(1)}(v)$ or $P_n^{(1)}(v)=v_1R_n^{(1)}(v)$.
 Since $$R_n^{(1)}(0,v_2,\ldots,v_n)=\sum_{|a|=n_{n,m_n}}v_2^{a_1}\cdots a_{n-1}^{a_{n-1}}h_{n,m_n,a}^{(1)}(v)$$  with $n_{n,m_n}\le m_n$. Thus, we have that $m_{\W_1}(R_n^{(1)}) \le m_n = m_{\W_0}(\fol_0)$ which results
\begin{equation}\label{for125}
m_{\W_1}(\fol_1)\le m_{\W_1}(v_1R_n^{(1)}(v))= 1 + m_{\W_1}(R_n^{(1)}) \le 1 +m_{\W_0}(\fol_0).
\end{equation}

Now, in the chart $\big((U_2)_1,v=\sigma_1(t))$, the foliation $\fol_2$ is described by the following vector field
\begin{equation}\label{for126}
X_2= P_1^{(2)}(t)\frac{\partial}{\partial t_1}+P_2^{(2)}(t)\frac{\partial}{\partial t_2}+\ldots+P_n^{(2)}(t)\frac{\partial}{\partial t_n}
\end{equation}
where $P_n^{(2)}(t) = R_n^{(2)}(t)$ or $P_3^{(2)}(t) =t_1 R_n^{(2)}(t)$ depending of the type that $\W_1$ is with

$$ R_n^{(2)}(t)=\sum_{j=m_n}^{\infty}\sum_{|a|=n_{n,m_n}^{(2)}}t_1^{a_2+\ldots+a_{n-1} +j-m_n-\beta_n}t_2^{a_2}\cdots t_{n-1}^{a_{n-1}}h_{n,j,a}^{(2)}(t),$$
where $\beta_n=\mbox{m}_{\W_1}(R_n^{(1)})$, $n_{n,m_n}^{(2)}=\mbox{m}_{\W_2}(R_n^{(2)})$ and $h_{n,j,a}^{(2)}(t)=h_{n,j,a}^{(1)}\circ\sigma_1$.

Let $\mathcal{M}= \{a=(a_1,\ldots,a_{n-1})| a_2+\ldots+a_{n-1} +j-m_n-\beta_n =0\}$ i.e.; if $a\in\mathcal{M}$ then $$a_2+\ldots+a_{n-1}=\beta_n-(j-m_n)\le \beta_n.$$ Again by hypothesis, there is the curve $\W_2$ which is defined as $$t_1 = t_2 - \psi_2^{(2)}(t_n)=\ldots=t_{n-1} - \psi_{n-1}^{(2)}(t_n)=0$$ contained in the singular set of $\fol_2$. But,
\begin{align*}
R_n^{(2)}(0,t_2,\ldots,t_n)&= \sum_{a\in\mathcal{M}}t_2^{a_2}\cdots t_{n-1}^{a_{n-1}}h_{n,j,a}^{(2)}(0,\ldots,0,t_n)\cr
&=\sum_{|a|=n_{n,m_n}^{(2)}}(t_2-\psi_2^{(2)}(t_n))^{a_2}\cdots (t_{n-1}-\psi_{n-1}^{(2)}(t_n))^{a_{n-1}}\widetilde{h}_{n,m_n,a}^{(2)}(t_2,\ldots,t_n)
\end{align*}
which results 
$$ \mbox{m}_{\W_2}(R_n^{(2)}) \le n_{n,m_n}^{(2)}\le m_n \le m_{\W_1}(R_n^{(1)}).$$
Therefore, 
\begin{equation}\label{for130}
m_{\W_2}(\fol_2) \le m_{\W_2}(P_n^{(2)})\le 1 + m_{\W_1}(R_n^{(2)})\le 1 + m_{\W_0}(\fol_0).
\end{equation}
Note that if $m_{\W_1}(P_n^{(2)}) > m_{\W_1}(P_1^{(2)})$ we can use the same arguments for $P_1^{(2)}=t_1R_1^{(2)}$ instead of $P_n^{(2)}.$

Continuing in this manner, we obtain:
 
\begin{equation}\label{for131}
m_{\W_i}(\fol_i) \le  1 + m_{\W_0}(\fol_0),\quad \forall i\ge 1.
\end{equation}
\end{proof}
\begin{theorem}
\label{theoremt} Let $\fol_0$ be a holomorphic foliation by curves
of degree $k$ defined on $\M_0=\P^3$ such that
$$\W_0\subset \sing\big(\fol_0\big)$$
where $\W_0$ is a smooth
curve with Euler characteristic $\chi(\W_0)$, degree $\deg(\W_0)$ and
a non-dicritical component of $\sing(\fol_0)$. Let
$\pi_1:\M_1\to\M_0$ be the blow-up of $\M_0$ along $\W_0$, with $\E_1=\pi_1^{-1}(\W_0)$ the exceptional divisor, $\fol_1$ the strict transform of $\fol_0$ under $\pi_1$ and
$\ell=m_{\E_1}(\pi_1^*\fol_0)$.
\begin{itemize}
\item [(a)] If $\W_0$ is of type II then the singular set of $\fol_1$ 
contains 
$$ \chi(\W_0)+(k-1)\deg(\W_0) +\ell\bigg(\chi(\W_0)-4\deg(\W_0)\bigg)/2$$
homeomorphic curves to $\P^1$, counting the multiplicities.

\item [(b)] If $\W_0$ is of type III then the singular set of $\sing(\fol_1)$ contains
$$2+\ell=1+m_{\W_0}(\fol_0)$$
homeomorphic curves to $\W_0$, counting the multiplicities.
\end{itemize}
\end{theorem}
\begin{proof}
The proof of Theorem \ref{theoremt} is identical to the proof of Theorem 4.7 in \cite{toulouse}. 
\end{proof}
\begin{remark} \rm If $\W_0$ is of type II, then there will necessarily be at least one homeomorphic curve to $\mathbb{P}^1$. Indeed, since $\W_0$ is a smooth curve, it is also a local complete intersection (l.c.i.), which implies that $\W_0$ can be locally defined by two polynomials $f_1=f_2=0$, where $d_j=\deg(f_j)$ with $d_1\le d_2$.
For the sake of contradiction, assume that $\W_0$ is of type II and there is no homeomorphic curve to $\mathbb{P}^1$ contained in $\text{Sing}(\mathcal{F}_1)$. Therefore, $\mathcal{F}_1$ can extend to a foliation on $\E_1$ without singularities on it, and its leaves are homeomorphic to $\W_0$. According to \cite{GM}, Proposition 2.3, and \cite{toulouse}, Theorem 4.6, we have $\chi(\W_0)=0$, which leads to $k=\deg(\mathcal{F}_0)=2\ell +1$, where $\ell=m_{\E_1}(\pi_1^*\mathcal{F}_0)$. However, this is impossible if $d_1\ge2$. In fact, if $d_1\ge 2$, then according to \cite{toulouse}, page 315, $\mathcal{F}_0$ is described by the following vector field: 
\begin{align*}
 X_0&= \bigg( \frac{\partial f_2}{\partial z_2}Q_1(z)-\frac{\partial f_1}{\partial z_2}Q_2(z)+\Delta_{23}Q_3\bigg)\frac{\partial}{\partial z_1}+\bigg( -\frac{\partial f_2}{\partial z_1}Q_1(z)+\frac{\partial f_1}{\partial z_1}Q_2(z)-\cr &-\Delta_{13}Q_3\bigg)\frac{\partial}{\partial z_2} +\Delta_{12}Q_3(z)\frac{\partial}{\partial z_3}
\end{align*}
where
$$ \Delta_{ij}=  \frac{\partial f_1}{\partial z_i} \frac{\partial f_2}{\partial z_j}- \frac{\partial f_1}{\partial z_j} \frac{\partial f_2}{\partial z_i},\quad Q_i(z) = \sum_{j=0}^{m_i}a_{ij}(z)\big(f_1(z)\big)^{m_i-j}\big(f_2(z)\big)^j.$$
Since $\W_0$ is of type II then $m_3+2 \le m_1=m_2$. Therefore,  
$$ deg\bigg(\frac{\partial f_2}{\partial z_i}Q_1(z)\bigg) \ge d_2 -1 + (\ell +2)d_1 + j(d_2-d_1)$$
for $j=0,\ldots,m_1$ since $a_{1j}(z)\not\equiv0$. Given that at least one $a_{1j}\not\equiv0$, we can conclude that
$$ deg(\fol_0) \ge 2\ell + 5.$$

On the other hand, if $d_1=1$ then $\W_0$ is a complete intersection and since $\fol_1$ extends to $\E_1$ without singularities we can consider $\fol_0$ to be special along $\W_0$. In order to have $\deg(\fol_0)=2\ell +1$, the unique possibility is $d_2\le 2$. See \cite{toulouse} for more details. But it implies $\chi(\W_0)=2$, which is absurd.
\end{remark}

\begin{example}\label{examp12}\rm 
Let $\mathcal{F}_0$ be a one-dimensional holomorphic foliation in $\M_0=\mathbb{P}^3$, such that $\W_0\subset \text{Sing}(\mathcal{F}_0)$, where $\W_0$ is a smooth curve.

In this example, we will assume that the desingularization process for $\W_0$ does not involve any dicritical curves. Thus, in some affine set, $\mathcal{F}_0$ can be described by a vector field $X_0$ as in (\ref{for114}). Let $\pi_1:\M_1\to\M_0$ be the blow-up of $\M_0$ along $\W_0$.

 If $\W_0$ is of type III, then there is at least one curve $\W_1\subset\text{Sing}(\mathcal{F}_1)$ such that $\W_1$ is homeomorphic to $\W_0$ with $\pi_1(\W_1)=\W_0$. See Theorem \ref{theoremt}. Thus, $\W_1$ can be locally defined as $u_1=u_2-\psi_1(u_3)=0$, which results in $\mathcal{F}_1$ is described by a vector field $X_1$ as in Equation (\ref{for124}). Let $\pi_2:\M_2\rightarrow\M_1$ be the blowup of $\M_1$ along $\W_1$, which is, by hypothesis, a non-dicritical curve of singularities. Given that the first and third sections of $X_1$ have $v_1$ as a factor, it is necessary for the curve $\W_2$, defined as $y_1=y_2=0$, where $\sigma_2(y)=v$, to be contained in $\text{Sing}(\mathcal{F}_2)$. Continuing in this manner, there exists a blow-up sequence $\{\pi_i,\M_i,\W_i,\fol_i,\E_i\}$ such that $\W_i$ is homeomorphic to $\W_{i-1}$ and $\pi_i(\W_i)=\W_{i-1}$ for all $i\ge1$.

If $\W_0$ is of type II, then $\mathcal{F}_1$ is described by a vector field $X_1$ as in Equation (\ref{exemp2}). In this situation, there is at least one curve $\W_1\subset\text{Sing}(\mathcal{F}_1)$ that is homeomorphic to $\mathbb{P}^1$, which can be locally defined as $u_1=u_3 - \psi_1(u_2)=0$. Consequently, from (\ref{exemp2}) the third section of $X_1$ can be written as follows:
$$Q_3(u)=\dps\sum_{j=m_3}^{\infty}v_1^{j-m_3}v_2^{n_{3,j}}h_{3,j}(v),\quad (v_1,v_2,v_3)=(u_1,u_3-\psi_1(u_2),u_2).$$
Hence, in this coordinate system $(v_1,v_2,v_3)$, the first and third sections of the vector field $X_1$ also have $v_1$ as a factor. Similarly to the previous case, there exists a blow-up sequence $\{\pi_i,\M_i,\W_i,\mathcal{F}_i,\E_i\}$ where $\W_i$ is homeomorphic to $\mathbb{P}^1$ and $\pi_i(\W_i)=\W_{i-1}$, but this time for all $i\ge2$.
\end{example}

\begin{theorem}\label{theoremd}

Let $\mathcal{F}_0$ be a one-dimensional holomorphic foliation defined on $\M_0=\mathbb{P}^n$ of degree $k$, and its singular locus contains a smooth curve $\W_0$ of degree $\deg(\W_0)$ and the Euler characteristic $\chi(W_0)$.

We will assume the existence of a finite blow-up sequence $\{\pi_i,\M_i,\W_i,\mathcal{F}_i,\E_i\}_{j=1}^{j-1}$ such that $\W_{i}\subset\E_i$ is homeomorphic to $\W_{i-1}$ with $\pi_i(\W_i)=\W_{i-1}$ and $\ell_i=m_{\E_i}(\pi_i^*\mathcal{F}_{i-1})$ for $i=1,\ldots,j$.

(a) If $\fol_{j-1}$ is special along $\W_{j-1}$ then 
$$\begin{array}{ll}\nn(\fol_j,\E_j)&=\dps\chi(\W_0)\sum_{i=0}^{n-2}(\ell_j+1)^i-\ell_j(1+\ell_j)^{n-2}\frac{\Lambda_0^{(n)}}{(n-1)^{j-1}}\cr
  &+(n-1)(\ell_j+1)^{n-2}\bigg((k-1)\deg(\W_0)-\Lambda_0^{(n)}\dps\sum_{i=1}^{j-1}\frac{\ell_i}{(n-1)^i}\bigg).
  \end{array}$$

(b) If $\sing(\fol_j)$ contains only isolated closed points on $\M_j$, then

    $$ \nn(\fol_j,\M_j)=\sum_{i=0}^{n}k^i + \sum_{m=1}^{j}\eta_m(\ell_1,\ldots,\ell_m)$$
    
    where 
    \begin{align}
      \eta_m&=\chi(\W_0)\bigg(\sum_{i=0}^{n-2}(1 +\ell_m)^i-(1+\ell_m)^{n-1}\bigg)+(1+\ell_m)^{n-2}\bigg((\ell_m^2-\ell_m)\frac{\Delta_0^{(n)}}{(n-1)^{m-1}}\bigg)\cr
            &-(1+\ell_m)^{n-2}(n\ell_m-n+2)\bigg((k-1)\deg(\W_0)\bigg)-\Delta_0^{(n)}\sum_{i=1}^{m-1}\frac{\ell_i}{(n-1)^i} \bigg)
    \end{align}

Here we are also assuming that $\dps\sum_{i=\alpha}^{\beta}a_i =0$ if $\alpha > \beta$.
\end{theorem}
 \begin{proof} For $j=1$, both formulas (a) and (b) align with Theorems 3.1 and 3.3 in \cite{toulouse}, respectively.  From Porteuos theorem ( see \cite{GH}, page 609) , we obtain 
\begin{equation}\label{ccab}
         \begin{array}{l}
          c_i(\tt_{\M_{j}}) = \pi_j^*c_i(\tt_{\M_{j-1}})+\dps\sum_{l=0}^{i-1}(-1)^{i-l-1}\bigg[\binom{n-1-l}{i-1-1}-\binom{n-1-l}{i-l}\bigg]\pi_j^*c_{l}(N_{j-1})\cdot\E_j^{i-l}\\
          +\dps\sum_{l=0}^{i-2}(-1)^{i-j}\bigg[\binom{n-1-l}{i-2-l}-\binom{n-1-l}{i-1-l}\bigg]\pi_j^*c_{l}(N_{j-1})\cdot\pi_j^*c_1(\tt_{\W_{j-1}})\cdot\E_j^{i-1-l}.
          \end{array}
\end{equation}
From \cite{bb}, it follows that  
$$ \tt_{\fol_{i}} \cong \pi_i^*\big(\tt_{\fol_{i-1}}\big) \otimes [\E_i]^{\ell_i}$$
which make us conclude 
\begin{equation}\label{crft}
c_1(\tt_{\fol_i}^*)=\pi_j^*c_1(\tt_{\fol_{i-1}}^*)-\ell_i\cdot\E_i,\quad\mathrm{ for}\quad i=1,\ldots,j.
\end{equation}
From (\ref{for22}) and (\ref{ccab}), we can show by finite induction that
\begin{align}\label{for50}
 \int_{\W_j}c_1(\tt_{\M_j})&=\int_{\W_0}c_1(\tt_{\M_0})-\dps\sum_{i=1}^{j}\int_{\W_i}(n-2)\E_i=(n+1)\deg(\W_0) -\bigg(1-\frac{1}{(n-1)^j}\bigg)\Lambda_0^{(n)}\cr
 &=(n+1)\deg(\W_0)-\Lambda_0^{(n)}+\frac{\Lambda_0^{(n)}}{(n-1)^j}\cr
\int_{\W_j}c_1(\tt_{\M_j}) &= \chi(\W_0) +\frac{\Lambda_0^{(n)}}{(n-1)^j},\quad j\ge0.
\end{align}
Similarly, we obtain
\begin{equation}\label{for51}
 \int_{\W_j}c_1(\tt_{\fol_j}^*)=(k-1)\deg(\W_0) -\Lambda_0^{(n)}\dps\sum_{i=1}^{j}\frac{\ell_i}{(n-1)^i},\quad j\ge0.
\end{equation}
Using the Baum-Bott's formula, we get that
$$ \nn(\fol_j,\E_j)=\int_{\E_j}c_{n-1}(\tt_{\E_j}\otimes \tt_{\fol_j}^*)=\int_{\E_j}\sum_{i=0}^{n-1}c_i(\tt_{\E_j})c_1^{n-1-i}(\tt_{\fol_j}^*).$$ 

From \cite{MAGR}, we get 
\begin{equation}\label{for52}
\begin{array}{ll}
     c_i(\tt_{\E_j})&=\dps\sum_{l=0}^{i-1}(-1)^{l}\pi_j^*c_{i-l}(\tt_{M_{j-1}})\zeta_j^{l} +(-1)^{i}\binom{n-1}{i}\zeta_j^i\cr
     &+\dps\sum_{l=1}^{i-1}\bigg( 1 - \binom{n-l-1}{i-l}\bigg)\pi_j^*c_l(N_{j-1})\zeta_j^{i-l}\cr
     &+\dps\sum_{l=0}^{i-2}\bigg( 1 - \binom{n-l-1}{i-l-1}\bigg)\pi_j^*c_l(N_{j-1})\pi_j^*c_1(\tt_{\W_{j-1}})\zeta_j^{i-l-1}     
\end{array}
\end{equation}
From (\ref{crft}), we have that 
\begin{equation}\label{for52*}
   c_1(\tt_{\fol_j}^*)^{n-1-i}=\sum_{m=0}^{n-1-i}(-1)^{n-1-i-m}\binom{n-1-i}{m}\pi_j^*c_1^m(\tt_{\fol_{j-1}}^*)\ell_j^{n-1-i-m}\E_j^{n-1-i-m}.
\end{equation}
Since that 
\begin{align*}
\int_{\E_j}\pi_j^*c_{i-l}(\M_{j-1})\pi_j^*c_1^m(\tt_{\fol_{j-1}}^*)\zeta_j^{n-i+l-m-1}&=0,\quad l\le i-2\quad\mathrm{or}\quad m\ge1 \cr
\int_{\E_j}\pi_j^*c_1^m(\tt_{\fol_{j-1}}^*)\zeta_j^{n-1-m}&=0\quad m\ge2\cr
\int_{\E_j}\pi_j^**c_l(N_{j-1})\pi_j^*c_1^m(\tt_{\fol_{j-1}}^*)\zeta_j^{n-1-l-m}&=0,\quad l+m\ge2\cr
\int_{\E_j}\pi_j^*c_l(N_{j-1})\pi_j^*c_1(\tt_{\W_{j-1}})\pi_j^*c_1^m(\tt_{\fol_{j-1}}^*)\zeta_j^{n-l-m-2}&=0,\quad l+m\ge1
\end{align*}
we get that 
\begin{align*}
 \nn(\fol_j,\E_j) &=\dps\sum_{i=0}^{n-1}\int_{\E_j}c_i(\tt_{\E_j}) c_1^{n-1-i}(\tt_{\fol_j}^*)= \cr
 &= \dps(-1)^{n-1}\sum_{i=0}^{n-1}\ell_j^{n-1-i}\int_{\E_j}\zeta_j^{n-1}+\sum_{i=1}^{n-1}(-1)^n\ell_j^{n-1-i}\int_{\E_j}\pi_j^*c_{1}(\tt_{M_{j-1}})\zeta_j^{n-2}\cr
 &+\dps\sum_{i=0}^{n-1}(-1)^n\binom{n-1}{i}\binom{n-1-i}{1}\ell_j^{n-i-2}\int_{\E_j}\pi_j^*c_1(\tt_{\fol_{j-1}}^*)\zeta_j^{n-2}\cr
 &+\dps\sum_{i=1}^{n-1}(-1)^{n-1}\bigg(1-\binom{n-2}{i-1}\bigg)\ell_j^{n-1-i}\int_{\E_j}\pi_j^*c_1(\nn_{j-1})\zeta_j^{n-2}\cr
& +\dps\sum_{i=2}^{n-1}(-1)^{n-1}\bigg(1-\binom{n-2}{i-1}\bigg)\ell_j^{n-1-i}\int_{\E_j}\pi_j^*c_1(\tt_{W_{j-1}})\zeta_j^{n-2}.\cr
\end{align*}
Since $c_1(\tt_{W_{j-1}})+c_1(N_{j-1})|_{\W_{j-1}}=c_{1}(\tt_{M_{j-1}}|_{\W_{j-1}}$, we have that

\begin{align*}
 \nn(\fol_j,\E_j) &= -\dps\sum_{i=0}^{n-1}\ell_j^{n-1-i}\int_{\W_{j-1}}c_1(\nn_{j-1})\cr
 &+\dps\sum_{i=0}^{n-1}\binom{n-1}{i}\binom{n-1-i}{1}\ell_j^{n-i-2}\int_{\W_{j-1}}c_1(\tt_{\fol_{j-1}})\cr
 &+\dps\sum_{i=1}^{n-1}\binom{n-2}{i-1}\ell_j^{n-1-i}\int_{\W_{j-1}}c_1(\nn_{j-1})\cr
& +\dps\sum_{i=1}^{n-1}\binom{n-2}{i-1}\ell_j^{n-1-i}\int_{\W_{j-1}}c_1(\tt_{W_{j-1}}).
\end{align*}
Therefore, 
\begin{align*}
 \nn(\fol_j,\E_j) &= \chi(\W_0)\sum_{i=0}^{n-2}(\ell_j+1)^i + (n-1)(\ell_j+1)^{n-2}(k-1)\deg(\W_0)\cr 
 &-(\ell_j+1)^{n-2}\Delta_0^{(n)}\bigg(\ell_j - (n-1)\dps\sum_{i=1}^{j-1}\frac{\ell_i}{(n-1)^i}\bigg).
\end{align*}
Thus, we obtain statement (a) of the Theorem.
Now, again by Baum-Bott's formula, 
$$\nn(\fol_j, \M_j) = \int_{\M_j}c_n(\tt_{\M_j}\otimes \tt_{\fol_j}^*)=\sum_{i=0}^{n}\int_{\M_j}c_i(\tt_{\M_j})c_1^{n-i}(\tt_{\fol_j}^*).$$
where
\begin{equation}\label{for179}
  c_1(\tt_{\fol_j}^*)^{n-i} = \sum_{m=0}^{n-i} (-1)^{n-i-m}\binom{n-i}{m}\pi_j^*c_1^{m}(\tt_{\fol_{j-1}}^*)\ell_j^{n-i-m}\E_j^{n-i-m}.
\end{equation}

Since $\dps\int_{\M_j}\E_j^n = \int_{\E_j}\zeta_j^{n-1}$ and 
\begin{align*}
  \int_{\M_j}\pi_j^*c_i(\tt_{\M_{j-1}})\pi_j^*c_1^{m}(\tt_{\fol_{j-1}}^*)\E_j^{n-i-m} &=0,\quad 2\le i+m < n\cr 
  \int_{\M_j}\pi_j^*c_l(\nn_{j-1})\pi_j^*c_1(\tt_{\W_{j-1}})\pi_j^*c_1^{m}(\tt_{\fol_{j-1}}^*)\E_j^{n-l-m-1} &=0,\quad l+m\ge1\cr
  \int_{\M_j}\pi_j^*c_l(\nn_{j-1})\pi_j^*c_1^{m}(\tt_{\fol_{j-1}}^*)\E_j^{n-l-m} &=0,\quad l+m\ge2
\end{align*}
we obtain from (\ref{ccab}) and (\ref{for179}) that 
\begin{align}\label{for55}
\nn(\fol_j, \M_j) &=\sum_{i=0}^{n}\int_{\M_{j}}\pi_j^*c_i(\tt_{\M_{j-1}})\pi_j^*c_1^{n-i}(\tt_{\fol_{j-1}}^*) +(-1)^{n-1}\int_{\M_j}\pi_j^*c_1(\tt_{\M_{j-1}})\ell_j^{n-1}\E_{j}^{n-1}\cr
                &+\dps\sum_{i=0}^{n}(-1)^{n-1}\bigg[\binom{n-1}{i-1}-\binom{n-1}{i}\bigg]\ell_j^{n-i}\int_{\M_j}\E_j^{n}\cr 
                &+\dps\sum_{i=0}^{n}(-1)^{n}\bigg[\binom{n-1}{i-1}-\binom{n-1}{i}\bigg]\binom{n-i}{1}\ell_j^{n-1-i}\int_{\M_j}\pi_j^*c_1(\tt_{\fol_{j-1}}^*)\cdot \E_j^{n-1}\cr
                &+\dps\sum_{i=2}^{n}(-1)^{n}\bigg[\binom{n-2}{i-2}-\binom{n-2}{i-1}\bigg]\ell_j^{n-i}\int_{\M_j}\pi_j^*c_1(\nn_{j-1})\cdot\E_j^{n-1}\cr
                &+\dps\sum_{i=2}^{n}(-1)^{n}\bigg[\binom{n-1}{i-2}-\binom{n-1}{i-1}\bigg]\ell_j^{n-i}\int_{\M_j}\pi_j^*c_1(\tt_{\W_{j-1}})\cdot\E_j^{n-1}                
\end{align}
Now, given that $\dps\int_{\M_{j}}\pi_j^*c_i(\tt_{\M_{j-1}})\pi_j^*c_1^{n-i}(\tt_{\fol_{j-1}}^*)=\int_{\M_{j-1}}c_i(\tt_{\M_{j-1}})c_1^{n-i}(\tt_{\fol_{j-1}}^*)$, so we get that  
\begin{align}\label{for56}
\nn(\fol_j, \M_j) &=\nn(\fol_{j-1},\M_{j-1})-\ell_j^{n-1}\int_{\W_{j-1}}c_1(\tt_{\M_{j-1}})\cr
                &-\dps\sum_{i=0}^{n}\bigg[\binom{n-1}{i-1}-\binom{n-1}{i}\bigg]\ell_j^{n-i}\int_{\W_{j-1}}c_1(\nn_{j-1})\cr 
                &+\dps\sum_{i=0}^{n}\bigg[\binom{n-1}{i-1}-\binom{n-1}{i}\bigg]\binom{n-i}{1}\ell_j^{n-1-i}\int_{\W_{j-1}}c_1(\tt_{\fol_{j-1}}^*)\cr
                &+\dps\sum_{i=2}^{n}\bigg[\binom{n-2}{i-2}-\binom{n-2}{i-1}\bigg]\ell_j^{n-i}\int_{\W_{j-1}}c_1(\nn_{j-1})\cr
                &+\dps\sum_{i=2}^{n}\bigg[\binom{n-1}{i-2}-\binom{n-1}{i-1}\bigg]\ell_j^{n-i}\int_{\W_{j-1}}*c_1(\tt_{\W_{j-1}})   
\end{align}
Thus, 
\begin{align}\label{for57}
\nn(\fol_j, \M_j) &=\nn(\fol_{j-1},\M_{j-1})-\ell_j^{n-1}\int_{\W_{j-1}}c_1(\tt_{\M_{j-1}})\cr
                &-(1+\ell_j)^{n-1}(1-\ell_j)\int_{\W_{j-1}}c_1(\nn_{j-1})\cr 
                &-(1+\ell_j)^{n-2}(n\ell_j-n+2)\int_{\W_{j-1}}c_1(\tt_{\fol_{j-1}}^*)\cr
                &+\big((1+\ell_j)^{n-2}(1-\ell_j)+\ell_j^{n-1}\big)\int_{\W_{j-1}}c_1(\nn_{j-1})\cr
                &+\bigg(\sum_{i=0}^{n-2}(1+\ell_j)^i - (1+\ell_j)^{n-1}+\ell_j^{n-1}\bigg)\int_{\W_{j-1}}c_1(\tt_{\W_{j-1}})   
\end{align}
Finally, from (\ref{for22}), (\ref{for50}) and (\ref{for51}), we conclude 
\begin{align}\label{for58}
\nn(\fol_j, \M_j) &=\nn(\fol_{j-1},\M_{j-1})+\chi(\W_0)\bigg(\sum_{i=0}^{n-2}(1+\ell_j)^i-(1+\ell_j)^{n-1}\bigg)\cr
                &-(1+\ell_j)^{n-2}(n\ell_j-n+2)\bigg((k-1)\deg(\W_0)-\Delta_0^{(n)}\sum_{i=1}^{j-1}\frac{\ell_i}{(n-1)^i}\bigg)\cr
                &+(1+\ell_j)^{n-2}\bigg((\ell_j^2-\ell_j)\frac{\Delta_0^{(n)}}{(n-1)^{j-1}}\bigg)
\end{align}
 As $\nn(\fol_0, \M_0)=\dps\sum_{i=0}^{n}k^i$ we conclude that
 
 $$\nn(\fol_j, \M_j)= \sum_{i=0}^{n}k^i + \sum_{m=1}^{j}\eta_i(\ell_1,\ldots,\ell_m).$$ 
\end{proof}
\begin{subsection}{ Proof of Theorem \ref{theoremB}}
From Lemma (\ref{lema4}), for each $j$, there exists a special holomorphic deformation $\mathcal{F}_{jt}$ of $\mathcal{F}_j$ for $0 < |t| <\epsilon_j$, where $\epsilon_j$ is sufficiently small. Thus, for $t\ne0$, the deformation $\mathcal{F}_{jt}$ is special along $\W_j$, and $m_{\E_j}(\mathcal{F}_{jt}) = m_{\E_j}(\mathcal{F}_j)=\ell_{j+1}$ or $\W_j$ is an invariant set of $\mathcal{F}_{jt}$. Furthermore, we can assume $\deg(\pi_j^*\cdots\pi_1^*\mathcal{F}_{jt})=\deg(\mathcal{F}_0)$. To achieve this, it suffices to consider that $\W_j$ is defined as $$u_1=u_2-\psi_2^{(j)}(u_n)=\ldots=u_{n-1}-\psi_{n-1}^{(j)}(u_n)=0,$$ and $\mathcal{F}_j$ is described by the following vector field:
\begin{equation}\label{for72}
          X_j = P^{(j)}_1(v)\frac{\partial}{\partial v_1}+ P^{(j)}_2(v)\frac{\partial}{\partial v_2}+\ldots+P^{(j)}_n(v)\frac{\partial}{\partial v_n}
\end{equation}
where $$v=(v_1,v_2,\ldots, v_n)=F_j(u)=(u_1, u_2-\psi_2^{(j)}(u_n)=\ldots=u_{n-1}-\psi_{n-1}^{(j)}(u_n),u_n).$$ Therefore, in this coordinate system,  the foliation $\fol_{jt}$ is described by the following vector field
\begin{equation}\label{for73}
          X_{jt}= X_j + t Y_j
\end{equation}
with the vector field $Y_j$ given as in Lemma (\ref{lema4}),  that is 
\begin{equation}\label{for74}
          Y_j=\sum_{i=1}^{n}Y_i\frac{\partial}{\partial v_i},\quad Y_i=\sum_{|a|=q_i}v_1^{a_1}\cdots v_{n-1}^{a_{n-1}}a_{i,a}(v), \quad q_i = m_{\W_j}(Y_i)
\end{equation}
where   
$$ q_1 =q_2=\ldots = q_{n-1}=q_n+1 = \ell_{j+1} +1.$$
Therefore, if $\ell_{j+1}=0$, i.e.; if $\W_j$ is of type III and $\mbox{m}_{\W_j}(\fol_j)=1$ then $\W_j$ is  an invariant set of $\mathcal{F}_{jt}$ for $0 < |t| <\epsilon_j$, since $q_n=0$. By Hartog's theorem,  the foliation $\fol_{jt}$ that  is generated by the vector $ F_j^*X_{jt}$ can be extend for whole $\M_j$. Furthermore, the coefficients $a_{ir}$ are chosen in order to have $\deg(\pi_j^*\cdot\cdot\cdot\pi_1^*\fol_{jt})=\deg(\fol_{0}), \forall t$. Varying the coefficients $a_{ir}$, if necessary, we can admit that $\sing(\fol_{jt})$ is composed of a curve $\W_j$ and some more isolated closed points $p_{s}^{(j)}$. Let $\widetilde{\fol_{jt}}$ be the strict transform of $\fol_{jt}$ under $\pi_{j+1}$. Therefore, we can determine the Milnor number $\mu(\fol_j,\W_j)$ as follows. 

If $\ell_{j+1}=0$ then
\begin{align}\label{for971}
\mu(\fol_j,\W_j) &=\int_{\W_j}c_1(\tt_{\W_j}\otimes \tt_{\fol_{jt}}^*)+N(\fol_j,\aa_{\W_j})\cr
                 &=\int_{\W_j}\bigg( c_1(\tt_{\W_j})+c_1(\tt_{\fol_{jt}}^*\bigg)+N(\fol_j,\aa_{\W_j})\cr
                 &=\chi(\W_0) + (k-1)\deg(\W_0)-\Delta_0^{(n)}\sum_{i=1}^{j}\frac{\ell_i}{(n-1)^i}+N(\fol_j,\aa_{\W_j}).
\end{align} 
Otherwise, since 
$$ \sing(\fol_{jt})=\W_j\cup \{ p_{s}^{(j)},s=1,\ldots,s_t\} ,  \sing(\widetilde{\fol_{jt}})=\{\widetilde{ p}_{r}^{(j)},r=1,\ldots,r_t\} $$
we get  
\begin{align}\label{91}
\mu(\fol_j,\W_j) &= \lim_{t\to0}\big( \mu(\fol_{jt},\W_j)\big) = \lim_{t\to0}\bigg(\nn(\fol_{jt},\M_j)-\sum_{p_{s}^{(j)}\notin A_{\W_j}}\mu(\fol_{jt},p_{s}^{(j)})\bigg)\cr
                           &=\nn(\fol_{j},\M_j)-\lim_{t\to0}\sum_{p_{s}^{(j)}\notin A_{\W_j}}\mu(\fol_{jt},p_{s}^{(j)})
\end{align}

On the other hand, since $\pi_{j+1}|_{\M_{j+1}\setminus\E_{j+1}}:=\pi_{j+1}:\M_{j+1}\setminus\E_{j+1}\to \M_j\setminus\W_j$ is an isomorphism, we get that
\begin{align}\label{for75}
\lim_{t\to0} \sum_{p_{s}^{(j)}\notin A_{\W_j}}\mu(\fol_{jt},p_{s}^{(j)})&=\lim_{t\to0}\sum_{\widetilde{p}_{r}^{(j)}\notin A_{\E_{j+1}}}\mu(\widetilde{\fol_{jt}},\widetilde{p}_{r}^{(j)})\cr
                                                         &=\nn(\fol_{j+1},\M_{j+1})-\nn(\fol_{j+1},\E_{j+1}) - N(\fol_j,\aa_{\W_j})
\end{align}
which results   
\begin{equation}\label{for92}
\mu(\fol_j,\W_j) = \nn(\fol_j,\M_j)-\nn(\fol_{j+1},\M_{j+1})+\nn(\fol_{j+1},\E_{j+1}) + N(\fol_j,\aa_{\W_j})
\end{equation}
where necessarily $N(\fol_j,\aa_{\W_j})\le N(\fol_{j-1},\aa_{\E_j}) = N(\fol_{j-1},\aa_{\W_{j-1}})$. From Theorem (\ref{theoremd}), we obtain the following
\begin{align}\label{for77}
 \mu(\fol_j,\W_j)=& -\eta_{j+1}(\ell_1,\ldots,\ell_{j+1}) +\nn(\fol_{j+1},\E_{j+1}) + N(\fol_j,\aa_{\W_j})\cr
                 =& (1+\ell_{j+1})^{n-1}\chi(\W_0)-\ell_{j+1}^2(1+\ell_{j+1})^{n-2}\frac{\Lambda_0^{(n)}}{(n-1)^j}+\cr
                 &+(n\ell_{j+1}+1)(1+\ell_{j+1})^{n-2}\bigg((k-1)\deg(\W_0)-\Lambda_0^{(n)}\sum_{i=1}^j\frac{\ell_i}{(n-1)^i}\bigg)+N(\fol_j,\aa_{\W_j})
\end{align}
Notice that for $j=0$, $\mu(\fol_0,\W_0)$  agrees with \cite{indices}. Furthermore, Clearly, (\ref{for92}) coincides with (\ref{for971}) since $\ell_{j+1}=0$. Again, given that
$\pi_{j+1}|_{\M_{j+1}\setminus\E_{j+1}}$ is a biholomorphism, we get $N(\fol_j,A_{\W_j})=N(\fol_{j+1},A_{\E_{j+1}})$ which results
\begin{align*}
\mu(\fol_{j+1},\bigcup_{i}\W_i^{(j+1)}) &= \nn(\fol_{j+1},\E_{j+1})+N(\fol_{j+1},A_{\E_{j+1}})\cr
               &=\mu(\fol_j,\W_j)+\nn(\fol_{j+1},\M_{j+1})-\nn(\fol_j,\M_j).
\end{align*}
Again, in the same way, we conclude the Item (b) of Theorem.
\end{subsection}

\begin{example}\label{examp11} \rm Let $\fol_0$ be the one-dimensional holomorphic
foliation of degree $4$ defined on $\M_0=\P^3$. This foliation is described in the open affine set  $U_3=\{[\xi_i]\in\P_3|\xi_3\neq0\}$ by the following vector
field
\begin{equation}\label{for219}
X_0=\big(P_3(z)+P_4(z)\big)\frac{\partial}{\partial z_1}+\big(Q_3(z)+Q_4(z)\big)\frac{\partial}{\partial
z_2}+\bigg(\sum_{i=1}^{4}R_i(z)\bigg)\frac{\partial}{\partial z_3}, 
\end{equation} where $z_i=\xi_{i-1}/\xi_3$ and
$$P_i(z)=\sum_{j=0}^{i}z_1^{i-j}z_2^jp_{ij}(z_3),
Q_i(z)=\sum_{j=0}^{i}z_1^{i-j}z_2^jq_{ij}(z_3),
R_i(z)=\sum_{j=0}^{i}z_1^{i-j}z_2^jr_{ij}(z_3)$$ with
$p_{ij},q_{ij},r_{ij}\in\C[z_3]$ are generic polynomials of degree $4-i$. Except for $q_{30}$ and $r_{10}$, which are identically null.

Thus, $\W_0=\{[\xi]\in\P^3|\xi_0=\xi_1=0\} \subset \sing(\fol_0)$ and is of type II. The singular set of $\fol_0$ contains another 36 closed points disjunct of $\W_0$. Thus, $\mu(\fol_0,\W_0) = 46$, which results in $N(\fol_0,\aa_{W_0})=21$. See Theorem (\ref{theorem1}).  Let $\pi_1:\M_1\to \M_0$ be the blowup of $\M_0$ along $\W_0$ and $\fol_1$ be the strict transform of $\fol_0$ under $\pi_1$. In the chart $\big((U_3)_1,\sigma_1(u)\big)$, the foliation $\fol_1$ is described by the following vector field
\begin{align}\label{for220}
X_1&=\big(u_1^2p_{30} + \widetilde{P_1}\big)\frac{\partial}{\partial u_1}+\big(u_1u_2(q_{31}-p_{30}) + u_1^2q_{40} +\widetilde{Q_1}\big)\frac{\partial}{\partial
u_2}\cr &+\bigg(u_2r_{11}+ u_1r_{20} + \widetilde{R_1}\bigg)\frac{\partial}{\partial u_3}
\end{align}
which $\W_1=\{u\in(U_3)_1|u_1=u_2=0\} \subset \sing(\fol_1)$ and $\mbox{m}_{\W_1}(\widetilde{P_1}),\mbox{m}_{\W_1}(\widetilde{Q_1})\ge3$ and $\mbox{m}_{\W_1}(\widetilde{R_1})\ge2$. In addition to $\W_1$, there are 4 other homeomorphic curves to $\P^1$ contained in $\E_1$ of which 3 are given by the roots of $r_{11}=r_{11}(u_3)$ and the fourth is defined as $\pi_1^{(-1)}\big([0:0:1:0]\big)$. Furthermore, $m_{\E_1}(\pi_1^*\fol_{0})=\ell_1=1$.

Thus, $\W_1$ is of type I and is homemorphic to $\W_0$ with $\pi_1(\W_1)=\W_0$.  Let $\pi_1:\M_2\to \M_1$ be the blow-up of $\M_1$ along $\W_1$ and $\fol_2$ be the strict transform of $\fol_1$ under $\pi_2$ with $m_{\E_2}(\pi_2^*\fol_{1})=\ell_2=1$. But, this time, in $\E_2$ there is no homeomorphic curve to $\W_1$. It is not difficult to see that $\widetilde{\fol_2}=\fol_2|_{\E_2}$ contains 12 closed points in its singular set. In order to verify Theorem (\ref{theoremd}), we need to make a small perturbation on $\fol_1$ because $\sing(\fol_1)$ contains four curves on $\E_1$.
Let $\fol_{1t}$ be the one-dimensional holomorphic foliation on $\M_1$ which is described in $(U_3)_1$ by following vector field
\begin{equation}\label{for221}
X_{1t}= X_1+ tY_1,
\end{equation}
where 
$$ Y_1 = \sum_{j=0}^{2}a_ju_1^{2-j}u_2^j\frac{\partial}{\partial u_1}+\sum_{j=0}^{2}b_ju_1^{2-j}u_2^j\frac{\partial}{\partial u_2}+\big(c_0u_1+c_1u_2\big)\frac{\partial}{\partial u_3}$$
with $X_1$ given in (\ref{for220}) and $a_i, b_i,c_i\in\C[u_3]$  are generic polynomials of degree $1+i, i$ and $2+i$, respectively. 

Therefore, $\fol_{1t}$ is a small perturbation of $\fol_1$, but now $\W_1$ is a special component of $\sing(\fol_{1t})$, $t\ne0$. However, in order for there to be another holomorphic family $\fol_{0t}$ such that $\fol_{1t}$ is the strict transform of $\fol_{0t}$ under $\pi_1$ then we must need to have $a_2\equiv0$ in (\ref{for221}). In fact, let $\fol_{0t}$ be the holomorphic family which is described on the affine open set $U_3$ by the following vector field

\begin{equation}\label{fexam1}
X_{0t}=X_0 +tY_0,
\end{equation}
where $X_0$ as in (\ref{for219}) and 
$$Y_0=(a_0z_1^3+a_1z_1z_2)\frac{\partial}{\partial z_1}+\bigg(b_0z_1^4 + (b_1+a_0)z_1^2z_2+(b_2+a_1)z_2^2\bigg)\frac{\partial}{\partial z_2}+(c_0z_1^2+c_1z_2)\frac{\partial}{\partial z_3}.$$

The foliation $\fol_{0t}$ also has degree 4 and $\W_0$ is a type I component for $\fol_{0t}$, with $t\ne0$. Furthermore,  $m_{\E_1}(\pi_1^*\fol_{0t})=\ell_1=1$ and $\fol_{1t}$ is the strict transform of $\fol_{0t}$ from $\pi_1$. 

Changing the coefficients of $a_i,b_i,c_i$ if necessary, we can admit that $\W_1$ is the unique curve of $\sing(\fol_{1t})$ which is also a special component of $\sing(\fol_{1t})$, $t\ne0$. In addition to $\W_1$, $\fol_{1t}$ has four more singularities denoted by $p_i^{(1t)}$ in $\E_1$ and $m_{\E_2}(\pi_2^*\fol_{1t})=\ell_2=1$ for all $t\ne0$. Therefore, from Theorems \ref{theorem1} and \ref{theoremB}, we get 
\begin{equation}\label{for711}
\mu(\fol_{1t},\W_1) + \sum_{i=1}^{4}\mu(\fol_{1t},p_i^{(1t)}) = \mu(\fol_{0t},\W_0) - 14 = 14+ N(\fol_{0t},A_{\W_0})
\end{equation}
for $t\ne0$. See also the Remark 3.4.  Given $\mu(\fol_{1t},\W_1) = 22$ since $\fol_{1t}$ is special along $\W_1$, we conclude
\begin{equation}\label{for712}
N(\fol_{0t},A_{\W_0}) = 8 + \sum_{i=1}^{4}\mu(\fol_{1t},p_i^{(1t)}) = 12.
\end{equation}
Therefore, for $t\ne0$, the singular set of $\fol_{0t}$ contains 36 + 9 =45 isolated closed points disjunct to $\W_0$.  Keeping this notation, let $\fol_{2t}$ the strict transform of $\fol_{1t}$ from $\pi_2$. Thus, it is not difficult to see that $\fol_{2t}$ contains 12 isolated singularities on $\E_2$, counting the multiplicities. Consequently, the singular set of $\fol_{2t}$ contains $45 + 4 +12=61$ isolated closed points, counted the multiplicities. These facts agree with Theorem \ref{theoremd}.

\end{example}

\section{ An application: Seidenberg's Theorem for non-isolated singularities}
\begin{lema}
\label{lemafp0} Let $\mathcal{F}_0$ be a one-dimensional holomorphic foliation on $\M_0=\mathbb{P}^n$ of degree $k$, and its singular locus contains a smooth curve $\W_0$ of degree $\deg(\W_0)$ and the Euler characteristic $\chi(\W_0)$.
We will assume the existence of a blow-up sequence $\{\pi_j,\M_j,\W_j,\mathcal{F}_j,\E_j\}$.

If each $\W_j$ is homeomorphic to $\W_{j-1}$ with $\pi_j(\W_j)=\W_{j-1}$ for all $j\ge 1$, then there exists a natural number $r$ such that $m_{\W_{r}}(\mathcal{F}_{r}) = 1$ and $\W_{r}$ is of type III.
\end{lema}
\begin{proof}
Through contradiction, let us assume that this theorem is false, i.e., $m_{\W_j}(\mathcal{F}_j) \ge 2$ or $\W_j$ is not of type III for all $j\ge 1$. From Equation (\ref{for131}), we can make the assumption:

\begin{equation}\label{for71}
          1 \le \ell_j = m_{\E_j}\big(\pi_j^*\fol_{j-1}\big) \le m_{\W_0}(\fol_0)+1,\quad\forall j\ge1.
\end{equation}

However, as $\mu(\mathcal{F}_j,\W_j)$ given in Theorem \ref{theoremB} is a natural number for all $j\ge0$, we can infer that the sequence:
\begin{equation}\label{for78}
a_j = (\ell_{j+1}+1)^{n-2}\Lambda_0^{(n)}\bigg(\frac{\ell_{j+1}^2}{(n-1)^j}+\big(n\ell_{j+1}+1\big)\dps\sum_{i=1}^{j}\frac{\ell_i}{(n-1)^i}\bigg) \in \N,\quad \forall j.
\end{equation}
On the other hand, the sequence of natural numbers ${\ell_j}$ is bounded, implying the existence of a subsequence ${\ell_{j_i+1}}$ where each $\ell_{j_i+1}=\ell$ is a constant for all $j_i$. Consequently,
\begin{equation}\label{for781}
a_{j_i} = (\ell+1)\Lambda_0^{(n)}\bigg(\frac{\ell^2}{(n-1)^{j_i}}+ \big(n\ell+1\big)\dps\sum_{i=1}^{j_i}\frac{\ell_i}{(n-1)^i}\bigg) \in \N,\quad \forall j_i.
\end{equation}

Therefore, 

$$0< |a_{j_{i_2}}-a_{j_{i_1}}| = (\ell+1)\bigg|\Lambda_0^{(n)}\bigg|\bigg( \frac{\ell^2}{(n-1)^{j_{i_1}}}+\frac{\ell^2}{(n-1)^{j_{i_2}}}+\big(n\ell+1\big)\sum_{r=j_{i_1}}^{_{j_{i_2}}}\frac{\ell_i}{(n-1)^i} \bigg).$$

For sufficiently large $j_{i_1}$ and $j_{i_2}$, we obtain $0<|a_{j_{i_2}}-a_{j_{i_1}}|<1$. However, this is absurd, considering that $a_j$ is a natural number for all $j$.
\end{proof}

\subsection{Maximum number of blowups for the desingularization}
Now, we calculate the maximum number of blowups  needed until we reach $\ell_i = m_{\E_i}\big(\pi_i^*\fol_{i-1}\big)=0$. In fact, we will suppose $\ell_j=\ell_1$ for $j=1,\ldots,N_1+1$. Then, from (\ref{for78}), we get

$$a_j = (\ell_{j+1}+1)^{n-2}\Lambda_0^{(n)}\bigg(\frac{\ell_{j+1}^2}{(n-1)^j}+\big(n\ell_{j+1}+1\big)\dps\sum_{i=1}^{j}\frac{\ell_i}{(n-1)^i}\bigg)$$
is a natural number for $j=1,\ldots,n_1$. Thus, 

$$b_j =a_j-a_{j-1}= \frac{\ell_1(1+\ell_1)^{n-2}(1+2\ell_1)\Lambda_0^{(n)}}{(n-1)^j}\in \N,\quad\text{for } j=1,\ldots,N_1.$$

But, we can consider $\ell_1(1+\ell_1)^{n-2}(1+2\ell_1)\Lambda_0^{(n)}=(n-1)^{\alpha_0}\beta_0$ such that $n-1$ does not divide $\beta_0$. Thus, we have the worst situation that $n_1 \le \alpha_0 \le \lfloor \log_{n-1}\big(\ell_1(1+\ell_1)^{n-2}(1+2\ell_1)\Lambda_0^{(n)}\big)\rfloor$.

Thus, we can now consider $\ell{n_1+1}>\ell_{n_1+2}=...=\ell_{n_2}=\ell_{n_2+1}$. Therefore, 

$$ b_j =a_j-a_{j-1}= \frac{\ell_{n_1+1}(1+\ell_{n_1+1})^{n-2}(1+2\ell_{n_1+11})\Lambda_0^{(n)}}{(n-1)^j}$$ 
is a natural number for $j=n_1+1,\ldots,n_2.$                

Consequently, 

$$n_2 <\lfloor \log_{n-1}\big(\ell_1(1+\ell_1)^{n-2}(1+2\ell_1)\Lambda_0^{(n)}\big)\rfloor.$$

Therefore, repeating the same argument, it will be necessary  
$$ \sum_{\ell=1}^{\ell_1}\lfloor \log_{n-1}\big(\ell_1(1+\ell_1)^{n-2}(1+2\ell_1)\Lambda_0^{(n)}\big)\rfloor$$
for we get $\ell_j = m_{\E_j}\big(\pi_j^*\fol_{j-1}\big)=0$.
\begin{example} \rm 
    Let $\fol_0$ be the one-dimensional holomorphic
foliation of degree $m$ defined on $\M_0=\P^3$ which is described in the open affine set  $U_3=\{[\xi_i]\in\P_3|\xi_3\neq0\}$ by the following vector
field
$$X_0=\big(z_1^m-z_1z_2^2z_3^{m-2}\big)\frac{\partial}{\partial z_1}+\big(z_1^m-z_2^3z_3^{m-2}\big)\frac{\partial}{\partial z_2} -z_2^2z_3^{m-1}\frac{\partial}{\partial z_3}$$
where $z_i=\xi_{i-1}/\xi_3$. We also assume that $m$ is large enough. The singular set of $\fol_0$
contains two curves $\W_0=\{[\xi]\in\P^3| \xi_0=\xi_1=0\}$ and $\W_0^{(1)}=\{[\xi]\in\P^3| \xi_0=\xi_2=0$. We will consider the blow-up sequence $\{\pi_j,\M_j,\W_j,\fol_j,\E_j\}$ where $\W_j$ is defined in the chart $\big((U_3)_j,\sigma_1(x)=z^{(j-1)}\big)$, with $z^{(0)}=z$, by equations $x_1=x_2=0$. In this chart, the strict transform $\fol_j$ from $\fol_{j-1}$ via $\pi_j$ is described by the following vector field
$$X_j=x_1\big(x_1^{m-2j-}-x_2^2x_3^{m-2}\big)\frac{\partial}{\partial x_1}+\big(x_1^{m-3j}(1-jx_1^{j-1}x_2)+(j-1)x_2^3x_3^{m-2}\big)\frac{\partial}{\partial x_2} -x_2^2x_3^{m-1}\frac{\partial}{\partial x_3}$$ for $j$ such that $m-3j\ge0$. Thus, the singular set of $\fol_j$ always contains two curves $x_1=x_2=0$ and $x_1=x_3=0$ if $m-3j\ge 1$. We will determine the Milnor numbers $\mu(\fol_j,\W_j)$ for some $j$. Firstly, for $j=0$, we can make the holomorphic perturbation $\fol_{0t}$ of $\fol_0$ which is described the following vector field
$$X_{0t}=X +t\bigg(\sum_{i=0}^{3}a_i(z_3)z_1^{3-i}z_2^i\frac{\partial}{\partial z_1}+\sum_{i=0}^{3}b_i(z_3)z_1^{3-i}z_2^i\frac{\partial}{\partial z_2}+\sum_{i=0}^{2}c_i(z_3)z_1^{2-i}z_2^i\frac{\partial}{\partial z_3}\bigg)$$
where $a_i,b_i,c_i$ are generic polynomials of degree $m-3,m-3$ and $m-2$, respectively. Thus, the singular set of $\fol_{0t}$ is composed by the curve $\W_0$ and some isolated closed points. It is not difficult to show Theorem \ref{theoremd}(a) agrees with number of isolated singularities of $\fol_{1t}$ on $\E_1$, where $\fol_{1t}$ is the strict transform of $\fol_{0t}$ from $\pi_1$ for $t\ne0$. From Theorems \ref{theorem1} and \ref{theoremB}, we can determine $\mu(\fol_0,\W_0)$. Now, we will determine $\mu(\fol_1,\W_1)$. Initially, we will consider the holomorphic perturbations of $\fol_1$ as follows
\begin{equation}\label{for345}X_{1t} = X_1 + t\bigg(\sum_{i=0}^{3}a_i(x_3)x_1^{3-i}x_2^i\frac{\partial}{\partial x_1}+\sum_{i=0}^{3}b_i(x_3)x_1^{3-i}x_2^i\frac{\partial}{\partial x_2}+\sum_{i=0}^{2}c_i(x_3)x_1^{2-i}x_2^i\frac{\partial}{\partial x_3}\bigg)
\end{equation}
where $a_i,b_i,c_i\in \C[z_3]$ of degree $m-5+i,m-6+i$ and $m-4+i$. Thus, $m_{\E_2}(\pi_1^*\fol_{1t})=2$ for all $t$. Again, Theorem \ref{theoremd}(a) agrees with number of isolated singularities of $\fol_{2t}$ on $\E_1$, where $\fol_{2t}$ is the strict transform of $\fol_{1t}$ from $\pi_2$  for $t\ne0$. However, the main problem of these perturbations is that there is no holomorphic foliation $\fol_{0t}$ such that $\fol_{1t}$ is the strict transform from $\fol_{0t}$ via $\pi_1$ because $a_3$ can be not identically null. But, if $a_3\equiv0$ then $\fol_{1t}$ is not special along $\W_1$. In fact, from (\ref{for345}),  there are embedding closed points associated to $\W_1$, given by $A^t_i(0,0,x^t_{3i})$ where $x^t_{3i}$ is a root of $t(b_3(x_3)-x_3^{m-2}=0$. Thus, in order to make a holomorphic deformations $\fol_{1t}$ such that $\W_1$ is special component we need to consider $m_{\E_2}(\pi_2^*\fol_{1t})=1$. More precisely, we will consider the following deformation $\fol_{1t}$  described by vector field
\begin{equation}\label{for346}X_{1t} = X_1 + t\bigg(\big(\sum_{i=0}^{2}b_i(x_3)x_1^{2-i}x_2^i+b_3(x_3)x_2^3\big)\frac{\partial}{\partial x_2}+\big(x_1c_0(x_3) + x_2c_1(x_3)\big)\frac{\partial}{\partial x_3}\bigg)
\end{equation}
where $b_i$ and $c_i$ are generic polynomials of degree $m-5+i$ and $m-3+i$, respectively, except for $b_3$ which also have degree equal to $m-3$. Now, $\fol_{1t}$ is special along $\W_1$ and there exists a one-dimensional holomorphic foliation $\fol_{0t}$ on $\M_0$ such that $\fol_{1t}$ is the strict transform of $\fol_{0t}$ from $\pi_1$. In fact, $\fol_{0t}$ is described by the following vector field
\begin{equation}\label{for347}X_{0t} = X_0 + t\bigg(\big(\sum_{i=0}^{2}b_i(z_3)z_1^{5-2i}z_2^i+b_3(z_3)z_2^3\big)\frac{\partial}{\partial z_2}+\big(z_1^3c_0(x_3) + z_1z_2c_1(z_3)\big)\frac{\partial}{\partial z_3}\bigg).
\end{equation}
Let $\fol_{2t}$ be strict transform from $\fol_{1t}$ via $\pi_2$. From Theorem \ref{theoremd}(a) we have that $\nn(\fol_{2t},\E_2)=4m-8$ for $t\ne0$, since $\ell_1=2$ and $\ell_1=1$. The same idea can be used to calculate $\mu(\fol_j,\W_j$ for $j>1$.
\end{example}
\subsection{Normal forms for non-isolated singularities}

We will focus on the germs of holomorphic foliations $\mathcal{F}_0$ defined in an open set $U_0\subset\mathbb{C}^3$ such that their singular set contains a smooth curve $\W_0$ of type III with $\mbox{m}_{\W_0}(\mathcal{F}_0)=1$. Without loss of generality, we can assume that $\W_0$ is defined as $z_0=z_1=0$. Keeping the notation, we consider the blow-up sequence $\{\pi_i,\M_i,\W_i,\mathcal{F}_i,\E_i\}$.

According to Theorem (\ref{theoremt}), there exist two curves in $\E_1$ that are homeomorphic to $\W_0$, counting multiplicities, since $\mbox{m}_{\E_1}(\pi_1^*\mathcal{F}_0)=0$. We will start by considering that $\mathcal{F}_0$ is described by the following vector field:
\begin{equation}\label{champ1}
X_0=\dps\sum_{i=1}^{3}(z_1p_{i0}(z_3) +z_2 p_{i1}(z_3) + P_i(z))\frac{\partial}{\partial z_i}, \quad \mbox{m}_{\W_0}(P_i)\ge2.
\end{equation}
By the way, from Equation (\ref{defmc}) we have 
\begin{equation}\label{matr}
\mathbf{A}_{X_0}|_{\W_0} = \begin{pmatrix} p_{10}& p_{11}
\cr
                                   p_{20}& p_{21}\end{pmatrix}.
                                   \end{equation}

Let $\lambda_1=\lambda_1(z_3)$ and $\lambda_2=\lambda_2(z_3)$ be the eigenvalues of the
matrix (\ref{matr}).Thus, we have the following three cases to consider
\begin{itemize}
\item [(i)] $\lambda_1\cdot\lambda_2\not\equiv0$,  $\lambda_1\not\equiv n\lambda_2$ and $\lambda_2\not\equiv n\lambda_1$, $n\in\N$; 
\item [(ii)] $\lambda_1\not\equiv0$ and $\lambda_2\equiv0$;
\item [(iii)] $\lambda_1\equiv\lambda_2\equiv0$;
\end{itemize}
\begin{proposition}\label{proed} If  the eigenvalues $\lambda_1$ and $\lambda_2$  of the matrix $\mathbf{A}_{X_0}|_{\W_0}$ given in (\ref{matr}) are non-identically null with $\lambda_1\not\equiv n\lambda_2$ and $\lambda_2\not\equiv n \lambda_1$, $\forall n\in \N$ then these conditions are invariant by blowups along curves $\W_i$ homeomorphic to $\W_0$ with $\pi_i(\W_i) = \W_{0}$.
\end{proposition}
\begin{proof}
The eigenvalues $\lambda_1$ and $\lambda_2$ of the matrix (\ref{matr}) are given by
\begin{equation}\label{autv}
\lambda_i=\dps\frac{p_{10}+p_{21}}{2}+(-1)^i
\frac{\sqrt{\Delta}}{2} \end{equation} where
$$\Delta=\Delta(z_3)= (p_{10}-p_{21})^2+4p_{11}p_{20}.$$ 

Thus, in the chart $((U_1)_1,\sigma_1(u))$, $\fol_1$ is described by the vector field $X_1$ as
follows
\begin{equation}\label{ft31}
X_1=\dps u_1Q_1(u)\frac{\partial}{\partial
u_1}+(Q_2(u)-u_2Q_1(u))\frac{\partial}{\partial
u_2}+u_1Q_3(u)\frac{\partial}{\partial u_3}
\end{equation}
where $$Q_i(u)=p_{i0}+u_2p_{i1}+u_1\widetilde{Q_i}(u),\quad p_{ij}=p_{ij}(u_3),\quad \forall i$$ for some functions $\widetilde{Q_i}$.
However, since $\W_0$ is of type III, meaning it is a non-dicritical component, we have:
$$Q_2(0,u_2,u_3)-u_2Q_1(0,u_2,u_3)=p_{20}+u_2(p_{21}-p_{10})-u_2^2p_{11}\not\equiv0.
$$
This implies that there are two curves $\W_i^{(1)}$ that are homeomorphic to $\W_0$, defined as follows:
$\W_i^{(1)}=\{u\in (U_1)_1|u_1=u_2-\psi_i(u_3)=0\}$ where
\begin{equation}\label{ece}
\dps\psi_i(u_3)=\frac{p_{21}-p_{10}+(-1)^i\sqrt{\Delta}}{2p_{11}}
\end{equation}
for $p_{11}\not\equiv0$. ( If $\dps\lim_{u_3\to \alpha}\psi_i(u_3)=\infty$ then in the other chart $\big((U_1)_2,\sigma_2(v)\big)$, the curve $\W_i^{(1)}$ is defined as $v_1-\widetilde{\psi}_i(v_3)=v_2=0$ with $\widetilde{\psi}_i(\alpha)=0$).  Let $F_i(u)=(u_1,u_2-\psi_i(u_3),u_3)=v\in\C^3$ be a local
biholomorphism. Now, the push-forward foliation
$F_{i_{*}}(\fol_1)=\mathcal{G}_i$ 
defined in $V_i$ is described by the
following vector field
$$ Y_1= \dps \big(v_1r_{10}+R_1\big)\frac{\partial}{\partial
v_1}+\big(v_1r_{20}+v_2r_{21}+R_2\big)\frac{\partial}{\partial
v_2}+\big(v_1r_{30}+R_3\big)\frac{\partial}{\partial v_3}$$
where $r_{ij}=r_{ij}(v_3)$ and $\mbox{m}_{\W^{(1)}_1}(R_i)\ge2$. But,
$r_{10}=p_{10}+\psi_1p_{11}=\lambda_1$ and $r_{21}=p_{11}(\psi_2-\psi_1)=\lambda_2-\lambda_1$.
Consequently, the matrix
\begin{equation}\label{matr4}
\mathbf{A}_{Y_1}|_{\W_1^{(1)}} = \begin{pmatrix} r_{10}& 0\cr
                                   r_{20}& r_{21}\end{pmatrix}
                                   \end{equation}
has $\lambda_{11}^{(1)}=\lambda_1$ and $\lambda_{12}^{(1)}=\lambda_2-\lambda_1$ as
eigenvalues which are distinct and non-identically null.  Similarly for the other curve
$\W_2^{(1)}$ whose matrix $\mathbf{A}_{Y_2}|_{\W_2^{(1)}}$ has
$\lambda_{12}^{(1)}=\lambda_2$ and $\lambda_{22}^{(1)}=\lambda_1-\lambda_2$ as eigenvalues.

If $p_{11}\equiv0$ then the eigenvalues of (\ref{matr}) are $\lambda_1=p_{10}$ and $\lambda_2=p_{21}$. In this way, $\W_1^{(1)}=\{u\in(U_1)_1|  u_1 = u_2 +\dps\frac{p_{20}}{p_{21}-p_{10}}=0\}$ is a curve of singularities with eigenvalues $\lambda_{11}^{(1)}=\lambda_1$ and $\lambda_{12}^{(1)}=\lambda_2-\lambda_1$ while $\W_2^{(1)}=\{v\in(U_1)_2|  v_1 =v_2 =0\}$ is the other curve with eigenvalues $\lambda_{21}^{(1)}=\lambda_1-\lambda_2$ and $\lambda_{22}^{(1)}=\lambda_2$. Note that it is impossible $p_{11}\equiv p_{21}-p_{10}\equiv0$ since $\lambda_1$ and $\lambda_2$ are distinct.
\end{proof}

Now, let us consider the second case where $\lambda_1\not\equiv 0$ and $\lambda_2\equiv 0$. This leads to $tr(\mathbf{A}_{{X_0}|_{\W_0}})= p_{10}+p_{21}=\lambda_1$ and $\det(\mathbf{A}_{{X_0}|_{\W_0}})= p_{10}p_{21}-p_{11}p_{20}=0$. Therefore, 

$$\frac{p_{10}(z_3)}{p_{11}(z_3)}=\frac{p_{20}(z_3)}{p_{21}(z_3)}=\varphi(z_3)\implies p_{i0}(z_3)=\varphi(z_3)p_{i1}(z_3).$$ 
In the chart $((U_1)_1,\sigma_1(u))$, $\fol_1$ is
described by the vector field $X_1$ as in (\ref{ft31}). But this time, we have $Q_i(u)=p_{i1}(z_3)\big(u_2+\varphi(u_3)\big)+u_1\widetilde{Q_i}(u)$ for $i=1,2$, and $Q_3(u)=p_{30}+u_2p_{31}+u_1\widetilde{Q_3}(u)$. Since  
$$Q_2(0,u_2,u_3)-u_2Q_1(0,u_2,u_3)=\bigg(p_{21}-u_2p_{11}\bigg)\bigg(u_2+\varphi(u_3)\bigg)$$
there are two homeomorphic curves to $\W_0$ defined when $p_{11}\ne0$ as follows
$\W^{(1)}_1=\big\{ u\in (U_1)_1|u_1=u_2+\varphi(u_3)=0\big\}$ and $\W^{(1)}_2=\big\{ u\in (U_1)_1|u_1=u_2-p_{21}/p_{11}=0\big\}$.

\begin{proposition}\label{pro1an} If the matrix (\ref{matr}) has only one eigenvalue that is not identically zero, denoted by $\lambda_1$, then curves $\W_i^{(1)}$ that are homeomorphic to $\W_0$ with $\pi_1(\W_i^{(1)})=\W_0$ are elementary components of $\mathcal{F}_1$. The eigenvalues of $\fol_1$ at
$\W^{(1)}_1$ are $\lambda_{11}^{(1)}=\lambda_1$ and
$\lambda_{21}^{(1)}\equiv0$ while the eigenvalues of $\fol_1$ at
$\W_2^{(1)}$ are $\lambda_{12}^{(1)}=\lambda_1$ and $\lambda_{22}^{(1)}=-\lambda_1$.
\end{proposition}
\begin{proof} In fact, let $F_1(u)=(u_2+\varphi(u_3),u_1,u_3)=(v_1,v_2,v_3)$ be
a local biholomorphism. Therefore, the vector field
$Y_1=\big(F_1\big)_*X_{1}$ can be expressed as follows
\begin{equation}\label{champ8}
Y_{1}=\dps v_1(v_2r_{10}+ R_1)\frac{\partial}{\partial v_1} +
(v_1r_{20} + v_2r_{21}+R_2)\frac{\partial}{\partial v_2} + v_1(r_{30}+ v_2r_{31}+R_3)\frac{\partial}{\partial
v_3} \end{equation} where $r_{ij}=r_{ij}(v_3)$ and $\mbox{m}_{\W^{(1)}_1}({R_i})\ge2$.
But,
$r_{10}=p_{11}$ and $r_{21}= p_{21}+p_{11}\varphi(v_3)=p_{21}+p_{10}=\lambda_1$.
Consequently, we have

$$\mathbf{A}_{Y_1}|_{\W_1^{(1)}} = \begin{pmatrix} 0& 0 \cr

                                   r_{20}& \lambda_1\end{pmatrix}.$$
 
Therefore, the matrix $\mathbf{A}_{Y_1}|_{\W_1^{(1)}}$ has only one eigenvalue not identically zero
$\lambda_1$. In order to analyze the other curve $\W_2^{(1)}$ it is sufficient to consider the local
biholomorphism
$F_2(u)=(u_1,u_2-\frac{p_{21}}{p_{11}}, u_3)=(v_1,v_2,v_3)$ defined
for $p_{11}\ne0$. As before, let $Z_1=\big(F_2\big)_*X_{1}$,
i.e.,
\begin{equation}\label{champ9}
Z_1=\dps v_1(s_{10}+v_2s_{21}+S_1)\frac{\partial}{\partial v_1} +
(v_1s_{20}+v_2s_{21}+S_2)\frac{\partial}{\partial v_2} + v_1(s_{30}+v_2s_{31}+S_3)\frac{\partial}{\partial
v_3} \end{equation} where $s_{ij}=s_{ij}(v_3)$ and $\mbox{m}_{\W^{(2)}_1}({S_i})\ge2$.
 But, $s_{10}=-s_{21}= \lambda_1(v_3)$. Therefore, we get
$$\mathbf{A}_{Z_1}|_{\W_2^{(1)}} = \begin{pmatrix} \lambda_1& 0 \cr
 
                                   s_{20}& -\lambda_1\end{pmatrix}$$

\noindent i.e.; the matrix $\mathbf{A}_{Z_1}|_{\W_2^{(1)}}$ has two distinct eigenvalues denoted by $\lambda_{12}^{(1)}=\lambda_1(v_3)$
and $\lambda_{22}^{(1)}=-\lambda_1(v_3)$ with  
$\lambda_{12}^{(1)}/\lambda_{22}^{(1)}=-1\not\in Q_+$ for almost all
points of $\W_2^{(1)}$.
\end{proof}

Finally, we will consider the third case where $\lambda_1$ and $ \lambda_2$ are identically null, i.e., 
$tr(\mathbf{A}_{X_0}|_{\W_0})=det((\mathbf{A}_{X_0}|_{\W_0})=0$.
Again, we have two distinct situations to consider
$p_{10}=0$ or not. If
$p_{10}=0$ then $p_{21}=0$ and
$p_{11}p_{20}=0$. Let us only consider the case
where $p_{11}=0$ because the other case is analogous. Therefore, in (\ref{champ1}), 
the multiplicity $\mbox{m}_{\W_0}(P_1)=m_1\ge 2$ and 
\begin{equation}\label{matr1}
\mathbf{A}_{X_0}|_{\W_0} = \begin{pmatrix} 0& 0 \cr
                                   p_{20}& 0\end{pmatrix}
                                   \end{equation}

In $\E_1$, there are two types of curves. The first type is given by the roots of $p_{20}$, which are homeomorphic to $\mathbb{P}^1$. The second type is homeomorphic to $\W_0$.

To simplify our analysis, we will address the first type of curves in the chart $((U_1)_1,\sigma_1(u))$, while the second type of curves will be handled in the chart $((U_1)_2,\sigma_2(v))$. This allows us to have the following:
proposition
\begin{proposition}\label{propSS} If in the matrix (\ref{matr}) the coefficients  $p_{10}$, $ p_{11}$ and $ p_{21}$ are identically null and $p_{20}$ is an affine function with no common root with $p_{31}$ then there exists a blow-up sequence $\{\pi_i,\M_i,\W_i,\fol_i,\E_i\}$ such that $\W_i =(\pi_i)^{-1}(q)$, $q\in\W_{i-1}$ and $\W_i$ is a non-elementary component of $\sing(\fol_i)$ for all $i\ge 1$.
\end{proposition}
\begin{proof}
In the chart $((U_1)_1,\sigma_1(u))$, the foliation $\fol_1$ is described by the following vector field
\begin{equation}\label{champ2}
X_{1}=\dps u_1^{m_1}Q_1(u)\frac{\partial}{\partial u_1} +
\big(p_{20}+u_1{Q_2}\big)\frac{\partial}{\partial u_2} +u_1\big(p_{30}+u_2p_{31}+u_1Q_3\big)\frac{\partial}{\partial
u_3}
\end{equation}
where $p_{ij}=p_{ij}(u_3)$ and for certain functions $Q_i$. Let $\W_1$ be the curve defined by $u_1 = u_3 -\beta=0$ where $\beta$ is the root of $p_{20}$. Therefore,  $\W_1=(\pi_1)^{-1}(0,0,\beta)$.  Without loss of generality, we can assume that $p_{20}(u_3)=\alpha u_3$ with $\alpha\ne0$, that is, $\beta=0$.   In these coordinates $(w_1,w_2,w_3)=(u_1,u_3,u_2)$ the vector field $X_1$ in (\ref{champ2}) is rewritten as follows 
\begin{align}\label{champ3}
Y_1=&\dps w_1^{m_1}R_1(w)\frac{\partial}{\partial w_1} +
w_1\big(p_{30}(w_2)+w_3p_{31}(w_2)+w_1R_2\big)\frac{\partial}{\partial w_2}\cr
        & +\big(\alpha w_2 + w_1R_3\big)\frac{\partial}{\partial
w_3}
\end{align}

But, the second section of $Y_1$ in (\ref{champ3}) is
\begin{align}\label{for200}
w_1\big(p_{30}(w_2)+w_3p_{31}(w_2)+ w_1R_2\big)&=w_1\big(p_{30}(0)+w_3p_{31}(0)\big) +\widetilde{R}_2\cr
                                                                           &= w_1r_{20}(w_3) + \widetilde{R}_2
\end{align}
where $r_{20}(w_3)= p_{30}(0)+w_3p_{31}(0)$ and $\mbox{m}_{\W_1}(\widetilde{R}_2)\ge2$. Furthermore, $p_{31}(0)\ne0$
by hypothesis. So, 
\begin{equation}\label{matr3}
\mathbf{A}_{Y_1}|_{\W_1} = \begin{pmatrix} 0& 0 \cr
                                   r_{20}& 0\end{pmatrix}
                                   \end{equation}
where $r_{20}$ is also an affine function. In this way, the third section of $Y_1$ is 
\begin{align}\label{for201}
\alpha w_2 + w_1R_3 &= w_1R_3(0,0,w_3) +\alpha w_2 + \widetilde{R_3}\cr
                                    & =w_1r_{30} + w_2r_{31} + \widetilde{R_3}
\end{align}
where $r_{30}(w_3)=R_3(0,0,w_3)$,
$r_{31}(w_3)\equiv\alpha\ne0$ and $\mbox{m}_{\W_1}(\widetilde{R_3})\ge2$.
Consequently, the vector field $Y_1$ possesses the same properties as $X_1$ because $r_{20}$ is an affine function without any common root with $r_{31}(w_3)=\alpha\ne 0$. Hence, we can continue the blow-up process indefinitely, leading to $\W_2=(\pi_2)^{-1}(0,0,\beta_1)$, where $\beta_1=-r_{30}(0)/r_{31}(0)$, and so forth.
\end{proof}
\begin{corollary}
Let us assume in the matrix (\ref{matr}) that $p_{10}$, $p_{11}$  and $p_{21}$ and identically null and $p_{20}(z_3)=z_3g(z_3)$.
If $g(0)$ and $p_{31}(0)$ are nonzero then there exists a blow-up sequence $\{\pi_i,\M_i,\W_i,\fol_i,\E_i\}$ such that $\W_i =(\pi_i)^{-1}(q)$, $q\in\W_{i-1}$ and $\W_i$ is a non-elementary component of $\sing(\fol_i)$ for all $i\ge 1$.
\end{corollary}
\begin{proof}
In the chart $((U_1)_1,\sigma_1(u))$, the foliation $\fol_1$ is described by the following vector field 
$$ X_1=u_1^{m_1}Q_1(u)\frac{\partial}{\partial u_1}+ \big(u_3g(u_3) + u_1Q_2(u)\big)\frac{\partial}{\partial u_2}+ u_1\big(p_{30}+u_2p_{31}+u_1Q_3(u)\big)\frac{\partial}{\partial u_2}$$
where $p_{ij}=p_{ij}(u_3)$ and $Q_i(0,u_2,u_3)=0$ for $i=2,3$. But, in these coordinates $(w_1,w_2,w_3)=(u_1,u_3,u_2)$ the foliation $\fol_1$ is described by vector field $Y_1$ as follows 
\begin{align*}Y_1=&\dps w_1^{m_1}R_1\frac{\partial}{\partial w_1} +
w_1\big(p_{30}(w_2)+w_3p_{31}(w_2) + w_1R_2\big)\frac{\partial}{\partial w_2}+\cr & +\big(w_2g(w_2)+ w_1R_3\big)\frac{\partial}{\partial
w_3}.
\end{align*}
But, the second section of $Y_1$ is
\begin{align}\label{for202}
w_1\big(p_{30}(w_2)+w_3p_{31}(w_2)+ w_1R_2\big) &= w_1(p_{30}(0)+w_3p_{31}(0)+\widetilde{R}_2(w)\cr
                                                                              &= w_1 r_{20}(w_3) + \widetilde{R}_2(w)
\end{align}
while its third section is
\begin{align}\label{for202a}
w_2g(w_2) + w_1R_3(w) &= w_1R_3(0,0,w_3) +w_2g(0) + \widetilde{R}_3(w)\cr
                                          &= w_1 r_{30}(w_3) + w_2r_{31}(w_3) + \widetilde{R}_2(w).
\end{align}
Therefore,  $r_{20}(w_3)=p_{30}(0)+w_3p_{31}(0)$ is an affine function with no commom root with $r_{31}(w_3)=g(0)\ne0$. Therefore, we can aplly Proposition \ref{propSS} again.

\end{proof}
\begin{proposition}\label{pro00} If in the matrix (\ref{matr}) the coefiecients  $p_{10}$, $ p_{11}$ and $ p_{21}$ are identically null and $p_{20}\not\equiv0$.  Then, for any blow-up sequence $\{\pi_i,\M_i,\W_i,\fol_i,\E_i)$ such that $\W_i$ is homeomorphic to $\W_{i-1}$ and $\pi_i(\W_{i})=\W_{i-1}$ there is a natural number $k\in\N$ such
that $\W_i$ is an elementary component of $\sing(\W_i)$ for $i\ge k$.
\end{proposition}
\begin{proof}
The vector field (\ref{champ1}) can be rewritten as follows
\begin{equation}
\label{champ13}
X_0=\dps P_1(z)\frac{\partial}{\partial z_1} +
\big(z_1p_{20}+P_2(z)\big)\frac{\partial}{\partial z_2} +\big(z_1p_{30}+z_2p_{31}+P_3(z)\big)\frac{\partial}{\partial
z_3}
\end{equation}
where $p_{ij}=p_{ij}(z_3)$ and
$$P_i(z)=\dps\sum_{j=0}^{m_i}z_1^{m_i-j}z_2^jP_{ij}(z)=z_2^{n_i}g_i(z)+z_1L_i(z),\quad m_i=\mbox{m}_{\W_0}(P_i)\ge 2$$
with $n_i\ge m_i$,  either $g_i(0,0,z_3):=p_i(z_3)\not\equiv0$ or $g_i\equiv0$, $\mbox{m}_{\W_0}(L_i)\ge m_i-1$ with
$L_i(0,1,z_3):=q_i(z_3)$ for all $i$. In addition, if $\mbox{m}_{\W_0}(L_i)\ge m_i$ then $n_i=m_i$ and $p_i\not\equiv0$.

From (\ref{champ13}), in the chart $((U_1)_2,\sigma_2(v))$ 
the foliation $\fol_1$ is described by the following vector field
\begin{equation}\label{champ34}
Y_{1}=\dps\big(-v_1^2p_{20}+R_1-v_1R_2\big)\frac{\partial}{\partial v_1} +
v_2\big(v_1p_{20}+R_2\big)\frac{\partial}{\partial v_2} +v_2\big(v_1p_{30}+p_{31}+R_3\big)\frac{\partial}{\partial
v_3}
\end{equation}
where $p_{ij}=p_{ij}(v_3)$, $R_i=v_2^{n_i-1}g_i^{(1)}(v)+v_1v_2^{m_i-1}L_i^{(1)}(v)$ with $g_i^{(1)}(0,0,v_3)=p_i(v_3)$ and $ L_i^{(1)}(0,0,v_3)=q_i(v_3)$. 

In this chart, the exceptional divisor $\E_1$ is defined by the equation $\{v_2=0\}$ 
and the non-elementary curve
$\W_1=\{v\in(U_1)_2|v_1=v_2=0\}\subset\sing(\fol_1)$ has multiplicity equal to 2. Besides that
$\mbox{m}_{\E_1}(\pi_1^*\fol_0)=0$. From this point onward, let us focus solely on the fibers $\pi_1^{-1}(0,0,z_3)$ for which $p_{20}(z_3)\ne0$. This is because the curves associated with these fibers have already been taken into account in Proposition \ref{propSS}.

To continue with our analysis, we need to make some considerations about the possible values of $p_{31}$. Specifically, if $p_{31}\equiv0$, then in (\ref{champ13}), there must be a function $g_i$ that is not identically zero for some $i$. Otherwise, the hypersurface $v_1=0$ would be entirely contained in the singular set of $X_1$.

Hence, let us consider the case where $\mbox{m}_{\W_1}(\fol_1)=1$ which results in $p_{31}\not\equiv0$ or $\mbox{m}_{\W_1}(R_1)=1$.

Therefore, if $p_{31}\not\equiv0$ and $\mbox{m}_{\W_1}(R_1)\ge2$ then $\W_1$ is of type I.  In this situation, the singular set of $\fol_2$
restricted to the exceptional divisor consists of the elementary
curve $\W_2=\{x\in(U_2)_1|x_1=x_2=0\}$, with $\sigma_1(x)=v$,  and some
closed points. The eigenvalues of $\fol_2$ at $\W_2$ are
$\lambda_{11}^{(2)}=-p_{20}(x_3)$ and
$\lambda_{21}^{(2)}=2p_{20}(x_3)$, i.e.
$\lambda_{11}^{(2)}/\lambda_{21}^{(2)}=-1/2\not\in Q_+$ for almost
all $x\in\W_2$. More precisely, $\lambda_{11}^{(2)}/\lambda_{21}^{(2)}=-1/2\not\in Q_+$ for all $x_3$ such that $p_{20}(x_3)\ne0$.

Now, if $\mbox{m}_{\W_1}(R_1)=1$, then it must be the case that $m_1=n_1=2$ and $p_1\not\equiv0$, resulting in $\W_1$ being of type III. Consequently, the singular set of $\fol_2$, when restricted to $\E_2$, contains a curve that is homeomorphic to $\W_1$, but with multiplicity equal to 2. In fact, in the chart $\big((U_2)_1,\sigma_1(t)\big)$, $\fol_2$ can be described by the following vector field:
\begin{align}\label{for200a}
X_{2}=&\dps t_1\bigg(-t_1p_{20}+R_1^{(2)}-t_1R_2^{(2)}\bigg)\frac{\partial}{\partial t_1} +
t_2\bigg(2t_1p_{20}+R_1^{(2)}+2t_1R_2^{(2)}\bigg)\frac{\partial}{\partial t_2}+\cr 
&+t_1t_2\bigg(t_1p_{30}+p_{31}+R_3^{(2)}\bigg)\frac{\partial}{\partial
t_3}
\end{align}
where
$$ R_i^{(2)}(t)=(t_1t_2)^{m_i-1}L_i^{(2)}(t)+t_1^{n_i-2}t_2^{n_i-1}g_i^{(2)}(t),\quad g_i^{(2)}(0,0,t_3)=p_i(t_3).$$
Thus, $\W_2=\{t\in(U_2)_1|t_1=t_2=0\}\subset\sing(\fol_2)$.  But, in this situation, the singular set of $\fol_3$ contains three elementary curves homemorphic to $\W_2$. In fact, on the chart $\big((U_3)_1,\sigma_1(x)=t\big)$, there are two curves $\W_1^{(3)}=\{t\in(U_3)_1|x_1=x_2=0\}$ and $\W_2^{(3)}=\{t\in(U_3)_1|x_1=x_2-\frac{3p_{20}}{2p_1}=0\}$ while on the other chart $\big((U_3)_1,\sigma_2(y)=t\big)$, there is the third curve $\W_3^{(3)}=\{y\in(U_3)_2|y_1=y_2=0\}$. The eigenvalues of $\fol_3$ along $\W_1^{(3)}$ are $\lambda_{11}^{(3)}=-p_{20}(x_3)$ and $\lambda_{12}^{(3)}=3p_{20}(x_3)$, along  $\W_2^{(3)}$ are   $\lambda_{21}^{(3)}=\frac{p_{20}(x_3)}{2}$ and $\lambda_{22}^{(3)}=-3p_{20}(x_3)$, at $\W_3^{(3)}$ are $\lambda_{31}^{(3)}=2p_1(x_3)-2p_{20}(x_3)$ and $\lambda_{32}^{(3)}=2p_{20}(x_3)-p_1(x_3)$. It is worth noting that the case $\lambda_{31}^{(3)}\equiv\lambda_{32}^{(3)}\equiv0$ is not possible.

Henceforth, we will exclusively examine the situation where $\mbox{m}_{\W_1}(\fol_1)\ge2$, which implies that $p_{31}\equiv0$ and $\mbox{m}_{\W_1}(R_1)\ge2$, leading to $\W_1$ being of type III.
Furthermore, if $\mbox{m}_{\W_1}(R_1)=2$ then $m_1=2$ or $n_1=3$ while if $\mbox{m}_{\W_1}(R_2)=1$ then $m_2=n_2=2$ and $p_2\not\equiv0$. Thus, in the chart $\big((U_2)_1,v=\sigma_1(x)\big)$, the foliation $\fol_2$ is described by the following vector field
\begin{align}\label{for201a}
X_{2}=&\dps x_1\big(-p_{20}+R_1^{(2)}-R_2^{(2)}\big)\frac{\partial}{\partial x_1} +
x_2\big(2p_{20}-R_1^{(2)}+2R_2^{(2)}\big)\frac{\partial}{\partial x_2}\cr
& +x_1x_2\big(p_{30}+R_3^{(2)}\big)\frac{\partial}{\partial
x_3}
\end{align}
where 
\begin{align*}R_1^{(2)}(t)=&x_1^{m_1-2}x_2^{m_1-1}L_1^{(2)}(x)+x_1^{n_1-3}x_2^{n_1-1}g_1^{(2)}(x),\cr
R_i^{(2)}(x)=&x_1^{m_i-1}x_2^{m_i-1}L_i^{(2)}(x)+x_1^{n_i-2}x_2^{n_i-1}g_i^{(2)}(x),\quad \mbox{for }i=2,3
\end{align*}
with $L_i^{(2)}(0,0,x_3)=q_i(x_3)$ and $ g_i^{(2)}(0,0,x_3)=p_i(x_3)$.

Here, we assume that $p_1\equiv0$ when $n_3\ge4$, and similarly $p_2\equiv q_2\equiv0$ when $\mbox{m}_{\W_1}(R_2)\ge2$.
 Thus, the singular set of $\fol_2$ is defined as follows
$$x_1 = x_2\bigg( 2p_{20}+(2p_2-q_1)x_2 -p_1x_2^2\bigg)=0.$$ Thus, if $p_1\not\equiv0$ and $\Delta_1= (2p_2-q_1)^2+8p_{20}p_1 \not\equiv0$ then we have 3 curves to consider:
$\W_1^{(2)}=\{x\in(U_2)_1| x_1=x_2=0\}$, $\W_2^{(2)}=\{x\in(U_2)_1| x_1=x_2-\psi_1(x_3)=0\}$ and $\W_3^{(2)}=\{x\in(U_2)_1| x_1=x_2-\psi_2(x_3)=0\}$ where
$$ \psi_i(x_3) = \frac{2p_2(x_3)-q_1(x_3) +(-1)^i\sqrt{\Delta_1}}{2p_1(x_3)}.$$
The eigenvalues of $\fol_2$ at $\W_1^{(2)}$ are $\lambda_{11}^{(2)}=-p_{20}$ and $\lambda_{12}^{(2)}=2p_{20}$, at $\W_2^{(2)}$ are $\lambda_{21}^{(2)}=p_{20}+\psi_1p_2$ and $\lambda_{22}^{(2)}=\psi_1\sqrt{\Delta_1}$, 
at $\W_3^{(2)}$ are $\lambda_{31}^{(2)}=p_{20}+\psi_2p_2$ and $\lambda_{32}^{(2)}=-\psi_2\sqrt{\Delta_1}$.

If $\Delta_1\equiv0$ and $p_1\not\equiv0$ then there are 2 homeomorphic curves to $\W_1$ in the singular set of $\fol_2$, namely $\W_1^{(2)}=\{x\in(U_2)_1| x_1=x_2=0\}$ and $\W_2^{(2)}=\{x\in(U_2)_1| x_1=x_2-\psi_1(x_3)=0\}$, but this last one has multiplicity equal to 2. The eigenvalues of $\fol_2$ are $\lambda_{21}^{(2)}=p_{20}+\psi_1p_2$ and $\lambda_{22}^{(2)}\equiv0$.

If $p_1\equiv0$ and $2q_2-q_1\not\equiv0$ then there are 3 elementary curves contained in $\E_2$, that is, $\W_1^{(2)}=\{x\in(U_2)_1| x_1=x_2=0\}$, $\W_2^{(2)}=\{x\in(U_2)_1| x_1=x_2-\frac{2p_{20}}{q_1-2p_2}=0\}$ and $\W_3^{(2)}=\{y\in(U_2)_1| y_1=y_2=\}$ where $\sigma_2(y)=v$. The eigenvalues of $\fol_2$ along $\W_2^{(2)}$ are $\lambda_{21}^{(2)}=\frac{q_1p_{20}}{q_1-2p_2}$ and $\lambda_{22}^{(2)}=-2p_{20}$, along $\W_3^{(2)}$ are $\lambda_{31}^{(2)}=q_1-2p_2$ and $\lambda_{32}^{(2)}=p_2$.

Now, if $p_1\equiv0$ and $2q_2-q_1\equiv0$ then there are 2 elementary curves contained in $\E_2$, the elementary curve $\W_1^{(2)}=\{x\in(U_2)_1| x_1=x_2=0\}$ and $\W_2^{(2)}=\{y\in(U_2)_1| y_1=y_2=0\}$, with multiplicity equal to 2.  The eigenvalues of $\fol_2$ along  $\W_2^{(2)}$ are $\lambda_{21}^{(2)}\equiv0$ and $\lambda_{22}^{(2)}=q_2$. Notice if $q_2\equiv0$ then $q_1\equiv0$ resulting in $\mbox{m}_{\W_1}(R_1) >3$ and $\mbox{m}_{\W_1}(R_2) >2$ which is absurd.

Now, we will consider $\mbox{m}_{\W_1}(R_1) =3$, $\mbox{m}_{\W_1}(R_2) =2$, and $\mbox{m}_{\W_1}(R_3) =1$. Hence, $m_3=n_3=2$ and $p_3\not\equiv0$. Therefore, the singular set of $\fol_2$ has 2 curves homemorphic to $W_1$, namely, $\W_1^{(2)}=\{x\in(U_2)_1| x_1=x_2=0\}$ as before and $\W_2^{(2)}=\{y\in(U_2)_2| y_1=y_2=0\}$ with multilplicity equal to 2. But this curve  $\W_2^{(2)}$ is non-elementary and of type I. Thus, let $\W_2=\W_2^{(2)}$. The singular set of $\fol_3$  contains the elementary curve $\W_1^{(3)}=\{t\in(U_3)_1| t_1=t_2=0\}$ and some isolated closed points. The eigenvalues of $\fol_3$ along $\W_1^{(3)}$ are $\lambda_{11}^{(3)}=-2p_{20}$ and $\lambda_{12}^{(3)}=3p_{20}$.

Finally, if a blow-up sequence $\{\pi_i,\M_i,\W_i,\mathcal{F}_i,\E_i\}$ exists such that $\W_i$ is homeomorphic to $\W_{i-1}$ and $\pi_i(\W_i)=\W_{i-1}$, with $\W_i\subset\E_i$ being a non-elementary curve of $\text{Sing}(\mathcal{F}_i)$ for all $i$, then, based on our previous observations, the only possibility is $\text{m}_{\W_1}(R_1) \ge 3$, $\text{m}_{\W_1}(R_2) \ge 2$, and $\text{m}_{\W_1}(R_3) \ge 2$, and these conditions must remain invariant under subsequent blow-ups. More precisely, foliation $\fol_2$, in chart $\big((U_2)_2,\sigma_2(y)\big)$, is described by the following vector field.

\begin{align}\label{for913}
Y_2 &= \big( -2y_1^2p_{20}+ R_1^{(2)}-2y_1R_2^{(2)}\big)\frac{\partial}{\partial y_1}+y_2 \big( y_1p_{20}+ R_2^{(2)}\big)\frac{\partial}{\partial y_2}+\cr &+y_2 \big( y_1p_{30}+ R_3^{(2)}\big)\frac{\partial}{\partial y_3}
\end{align}
where 
$$ R_1^{(2)}=y_1y_2^{m_1-2}L_1^{(2)}(y)+y_2^{n_1-3}g_1^{(2)}(y), \quad R_i^{(2)}=y_1y_2^{m_i-1}L_i^{(2)}(y)+y_2^{n_i-2}g_i^{(2)}(y)$$
for $i=2,3$, with $L_i^{(2)}(y)=L_i^{(1)}\circ \sigma_2(y)$ and $g_i^{(2)}(y)=g_i^{(1)}\circ \sigma_2(y)$.

Thus, $\W_2=\{y\in (U_2)_2| y_1=y_2=0\}$ is a non-elementary and $\mbox{m}_{\W_2}(\fol_2) = 2$. In order for such a sequence to exist we must have $\mbox{m}_{\W_2}(R_1^{(2)}) \ge3$, $\mbox{m}_{\W_2}(R_2^{(2)}) \ge 2$ and $\mbox{m}_{\W_2}(R_3^{(2)}) \ge2$ and so on.  Therefore, let $\W_k$ be the curve defined in the chart $\big((U)_k,\sigma_2(y^{(k)})\big)$ as follows $\{ y^{(k)}\in(U_k)_2| y_1^{(k)}= y_2^{(k)}=0\}$, with $\sigma_2(y^{(k)})=y^{(k-1)}$ for $k\ge1$. By finite induction,
the foliation $\fol_k$, in the chart $\big((U_k)_1,\sigma_1(x)=y^{(k-1)})$, is described by the vector field
\begin{flalign}
\label{champ33}
 X_k&=\dps x_1\bigg(-(k-1)(p_{20}+S_2^{(k)})+ S_1^{(k)}\bigg)\frac{\partial}{\partial x_1} +
x_2\bigg(k(p_{20}+S_2^{(k)}) -S_1^{(k)}\bigg)\frac{\partial}{\partial x_2^{(k)}}+\cr
&+x_1x_2\bigg( p_{30} + S_3^{(k)}\bigg)\frac{\partial}{\partial x_3^{(k)}}
\end{flalign}
where 
\begin{flalign*}
S_1^{(k)}(x)&=x_1^{\alpha_k}x_2^{\alpha_{k-1}}\tilde{g}_1^{(k)}(x)+x_1^{m_1-k}(x_2)^{m_1-k+1}\tilde{L}_1^{(k)}(x),\quad \alpha_k=n_1-(2k-1)\cr
S_i^{(k)}(x)&=x_1^{n_i-k}x_2^{n_i-k+1}\tilde{g}_i^{(k)}(x)+(x_1x_2)^{m_i-1}\tilde{L}_i^{(k)}(x),\mbox{ for }i=2,3
\end{flalign*}
with $\tilde{L}_i^{(k)}(y)=L_i^{(k-1)}(\sigma_1(x))$ and $\tilde{g}_i^{(k)}(y)=g_i^{(k-1)}(\sigma_1(x))$, 
while in the chart $\big((U_k)_2,\sigma_2(y)=y^{(k-1)})\big)$, by the vector field
\begin{flalign}
\label{champ35} Y_k&=\dps \bigg(-ky_1^2p_{20}+R_1^{(k)}-ky_1R_2^{(k)}\bigg)\frac{\partial}{\partial y_1} +
y_2\bigg(y_1p_{20} + R_2^{(k)}\bigg)\frac{\partial}{\partial y_2}+\cr
&+y_2\bigg(y_1p_{30}+R_3^{(k)}\bigg)\frac{\partial}{\partial y_3}
\end{flalign}
where 
\begin{flalign*}
R_1^{(k)}(y)&=y_2^{\alpha_k}g_1^{(k)}(y)+y_1y_2^{m_1-k}L_1^{(k)}(y),\cr
R_i^{(k)}(y)&=y_2^{n_i-k}g_2^{(k)}(y)+y_1y_2^{m_i-1}L_2^{(k)}(y),\mbox{ for }i=2,3
\end{flalign*}
with $L_i^{(k)}(y)=L_i^{(k-1)}\circ\sigma_2(y)$ and $g_i^{(k)}(y)=g_i^{(k-1)}\circ \sigma_2(y)$.
Furthermore, $g_i^{(k)}(0,0,x_3)=p_i(x_3)$ and $L_i^{(k)}(0,0,x_3)=q_{i}(x_3)$ for all $i$. 

Since $\mbox{m}_{\W_{k-1}}(R_1^{(k-1)}) \ge3$, $\mbox{m}_{\W_{k-1}}(R_2^{(k-1)}) \ge 2$ and $\mbox{m}_{\W_{k-1}}(R_3^{(k-1)}) \ge2$, $\fol_k$ is well defined as well as
$\alpha_k,n_2-k,n_3-k,m_1-k\ge0$ with
$\mbox{m}_{\E_i}(\pi_i^*\fol_{i-1})=1$ for $i=1,\ldots,k$.  By Theorem \ref{theoremt} there are three curves $\W_j^{(k)}\subset\sing(\fol_k)$, counting the multiplicities, since $m_{\W_k}(\fol_{k})=2$. It is not
difficult to see that the curve
$\W_1^{(k)}=\{x^{(k)}\in(U_k)_1|x_1^{(k)}=x_2^{(k)}=0\}$ is an elementary component of
$\sing(\fol_k)$ since $\lambda_{11}^{(k)}=-(k-1)p_{20}(x_3)$ and
$\lambda_{21}^{(k)}=kp_{20}(x_3)$, i.e.;
$\lambda_{11}^{(k)}/\lambda_{21}^{(k)}=-(k-1)/k\not\in Q_+, k\ge2$  for almost
all $x^{(k)}\in\W_1^{(k)}$. Notice that $S_i^{(k)}|_{\E_k}\equiv0$ for all $i$ resulting in $\W_1^{(k)}$ is the unique homeomorphic curve to $\W_{k-1}$ contained in this chart.

However, considering that at least one of $p_i\not\equiv0$, it follows that at least one of the three inequalities $\text{m}_{\W_k}(R_1^{(k)}) \le \text{min}\{ m_1-k+1,\alpha_k\}$, $\text{m}{\W_{k}}(R_2^{(k)}) \le n_2-k$, or $\text{m}_{ \{W_{k}}(R_3^{(k)}) \le n_3-k$ holds true for all $k >0$. Consequently, it is impossible for such a sequence to continue indefinitely. This assumption is consistent since it would contradict Lemma \ref{lemafp0} otherwise.  Therefore, for some $k>0$, the analysis of Equation (\ref{champ35}) follows a pattern similar to that of Equation (\ref{champ34}), with the substitution of $p_{20}$ and $p_2$ by $kp_{20}$ and $kp_2$ in $R_1$, respectively. It is enough to compare Equations (\ref{champ34}) and (\ref{champ35}). Thus, the study is reduced to one of the cases that had previously been examined.
\end{proof}

\begin{proposition}\label{pro01}
Let us consider the foliation $\fol_0$ described by the vector field (\ref{champ1}) with  $p_{10}\not\equiv0$ and
the all eigenvalues of (\ref{matr}) are identically null. Then, for any blow-up sequence Let $\{\pi_i, \M_i,\W_i,\fol_i,\E_i\}$ be such that $\W_i$ is homeomorphic to $\W_{i-1}$ and $\pi_i(\W_i)=\W_{i-1}$  there is a natural $k\in\N$ such that $\W_i$ are elementary components of $\fol_i$ for $i\ge k$.
\end{proposition} . 
\begin{proof}
From (\ref{champ1}), we get  
\begin{equation}\label{champ50}
X_{0}=\dps \sum_{i=1}^{3}\bigg(z_1 p_{i0}+z_2p_{i1} + P_i(z)\bigg)\frac{\partial}{\partial z_i},
\end{equation} where
$$P_i(z)=\sum_{j=0}^{m_i}z_1^{m_i-j}z_2^jP_{ij}(z),\quad m_i \ge 2.$$

Under these conditions, since $tr(\mathbf{A}_{X_0}|_{\W_0})\equiv
det(\mathbf{A}_{X_0}|_{\W_0})equiv0$ there exists a holomorphic
function $\varphi(z_3)$ such that
$p_{i0}(z_3)=\varphi(z_3)p_{i1}(z_3)$ for $i=1,2$, which results
\begin{equation}
\mathbf{A}_{X_0}|_{\W_0} = \begin{pmatrix} \varphi(z_3)p_{11}(z_3)&p_{11}(z_3) \cr
                                  -\varphi^2(z_3) p_{11}(z_3)&- \varphi(z_3)p_{11}(z_3)\end{pmatrix}.
                                   \end{equation}
In the chart $(({U_1})_1,\sigma_1(u))$, the foliation $\fol_1$ is
described by the following vector field
\begin{flalign}\label{champ5}
X_1&=\dps u_1\bigg((u_2+\varphi(z_3))p_{11}+u_1Q_1\bigg)\frac{\partial}{\partial u_1} -
\bigg((u_2+\varphi(z_3))^2p_{11}-u_1(Q_2-u_2Q_1)\bigg)\frac{\partial}{\partial u_2}+\cr &+ u_1(p_{30} +u_2p_{31}+u_1Q_3)\frac{\partial}{\partial
u_3}
\end{flalign}
where 
$$Q_i(u)=u_1^{m_i-2}\sum_{j=0}^{m_i}u_2^jP_{ij}(u_1,u_1u_2,u_3).$$

In this chart, there is the only non-elementary curve
$\W_1$  which is defined by Equations $u_1=u_2+\varphi(u_3)=0$ and has multiplicity
equal to 2. Just as it was done
in the Proposition \ref{pro00}, from now on we only consider the
fibers $\pi_1^{-1}(0,0,z_3)$ such that $p_{11}(z_3)\ne0$. In this coordinate system  $(v_1,v_2,v_3)=F(u)=(u_1,u_2+\varphi(u_3),u_3)$,  the foliation $\fol_1$  is described by the vector field
\begin{flalign}
\label{champ6} Y_{1}&=\dps v_1\big(v_2 p_{11} + v_1R_1\big)\frac{\partial}{\partial v_1} -
\bigg( p_{11}v_2^2 - v_1\big(G_2 - v_2R_1+\varphi^{\prime}(v_3)( r_{30}+v_2r_{31})\big)\bigg)\frac{\partial}{\partial v_2}+\cr  &+v_1\big( r_{30}+v_2r_{31}+v_1R_3\big)\frac{\partial}{\partial
v_3},
\end{flalign}
where $R_i(v)=Q_i\circ F^{-1}(v)$ and $G_2(v) = R_2(v)+\varphi(v_3)R_1(v) + v_1\varphi^{\prime}(v_3)R_3(v)$. 

Initially, we will consider the case where $\varphi$ is not a constant function and  and $\mbox{m}_{\W_1}(\fol_1)=1$. Hence,  under these conditions, $r_{30}\not\equiv0$ or $a_0(v_3) := G_2(0,0,v_3)+\varphi^{\prime}(v_3)r_{30}(v_3)\not\equiv0$. 

Let us consider $a_0 \not\equiv 0$ which results $\W_1$ is of type III. So, in the chart $\big((U_2)_2,  v=\sigma_2(t)\big)$, the singular set of $\fol_2$ contains the only curve $\W_2=\{t\in(U_2)_2|t_1=t_2=0\}$ that is homeomorphic to $\W_1$. The curve $\W_2$ has multiplicity equal to 2 and $\mbox{m}_{\W_2}(\fol_2)=2$.  The singular set of $\fol_3$ contains three elementary homeomorphic curves to $\W_2$. In other words, in the chart $(U_3)_1,t=\sigma_1(x))$, there are curves $\W_1^{(3)}=\{x\in(U_3)_1| x_1=x_2=0\}$,  with eingevalues of $\fol_3$ at $\W_1^{(3)}$ are $\lambda_{11}^{(3)}=-a_0(x_3)$ and $\lambda_{11}^{(3)}=a_0(x_3)$; and $$\W_2^{(3)}=\left\{x\in(U_3)_1| x_1=x_2-\frac{2a_0(x_3)}{3p_{11}(x_3)}=0\right\}$$ with eingevalues of $\fol_3$ at $\W_2^{(3)}$ are $\lambda_{21}^{(3)}= a_0(x_3)/3$ and $\lambda_{22}^{(3)}=-2a_0(x_3)$. In the chart $(U_3)_2,t=\sigma_2(y))$, there is the curve $\W_3^{(3)}=\{y\in(U_3)_2| y_1=y_2=0\}$, with eingevalues of $\fol_3$ at $\W_3^{(3)}$ are $\lambda_{31}^{(3)}=3p_{11}(y_3)$ and $\lambda_{32}^{(3)}=-p_{11}(y_3)$.

Now, we will consider $a_0\equiv0$ and $r_{30}\not\equiv0$ which result $\W_1$ is of type I. Therefore, the singular set of $\fol_2$ contains the only curve $\W_2=\{y\in(U_2)_2|y_1=y_2=0\}$, $\sigma_2(y)=v$, which is homeomorphic to $\W_1$. The eigenvalues of $\fol_2$ at $\W_2$ are $\lambda_{21}^{(2)}=2p_{11}(y_3)$ and $\lambda_{21}^{(2)}=-p_{11}(y_3)$ whose ratio is a negative rational number for almost every point of $\W_2$.

Thus, we will consider $\mbox{m}_{\W_1}(\fol_1)=2$ which results $r_{30}\equiv0$ and $a_0\equiv0$ in (\ref{champ6}). Consequently, the singular set of $\fol_2$ contains three elementary homeomorphic curves to $\W_1$. In the chart $\big((U_2)_1,v=\sigma_1(x)\big)$, the singular set of $\fol_2$ is defined by the following Equations
\begin{equation}\label{for911}
x_1=-2x_2^2p_{11}(x_3) + x_2b_1(x_3) +a_1(x_3)=0
\end{equation}
where 
$$ a_1(x_3)=\frac{\partial G_2}{\partial v_1}(0,0,x_3)\mbox{ and } b_1(x_3)=\frac{\partial G_2}{\partial v_2}(0,0,x_3)-2R_1(0,0,v_3) + \varphi^{\prime}(x_3)r_{31}(x_3).$$

Hence, if $\Delta=b_1^2(x_3) + 8p_{11}(x_3)a_1(x_3)\not\equiv0$ then in this chart there are two curves which are defined as follows $\W_1^{(2)}=\{x\in(U_2)_1| x_1 =x_2 -\psi_1(x_3)=0\}$ and $\W_2^{(2)}=\{x\in(U_2)_1| x_1 =x_2 -\psi_2(x_3)=0\}$ where

$$ \phi_i(x_3) = \frac{b_1(x_3) -(-1)^i \sqrt{\Delta}}{4p_{11}(x_3)}.$$
The third curve is $\W_3^{(2)}=\{y\in(U_2)_2| y_1 =y_2 =0\}$, with $\sigma_2(y)=v$. The eigenvalues of $\fol_2$ along  $\W_1^{(2)}$ are $\lambda_{11}^{(2)}=\phi_1(x_3)p_{11}(x_3)+ R_1(0,0,x_3)$ and $\lambda_{12}^{(2)}=\sqrt{\Delta}$. The eigenvalues of $\fol_2$ along  $\W_2^{(2)}$ are $\lambda_{21}^{(2)}=\phi_2(x_3)p_{11}(x_3) R_1(0,0,x_3)$ and $\lambda_{22}^{(2)}=\sqrt{\Delta}$. And the eigenvalues of $\fol_2$ along  $\W_3^{(2)}$ are $\lambda_{31}^{(2)}=2p_{11}(x_3)$ and $\lambda_{12}^{(2)}=-p_{11}(x_3)$.

If $\Delta\equiv0$ and $b_1\not\equiv0$ in Equation (\ref{for911}) then the curve $\W_1^{(2)}$ has multiplicity equal to 2. Hence, $\W_1^{(2)}$ is an elementar component if $\lambda_{11}^{(2)}=\phi_1(x_3)p_{11}(x_3)+ R_1(0,0,x_3)\not\equiv0$. Otherwise, it is enough to make this change of variables $(t_1,t_2,t_3)=(x_2 - \psi_1(x_3),x_1,x_3$ which results in the vector field (\ref{champ6}) being transformed into the vector field (\ref{for913}). The proof then proceeds similarly to that of Proposition (\ref{pro00}).

In an exact similar way, if $\varphi$ is not a constant function or $\Delta\equiv b_1\equiv0$, the vector field (\ref{champ6}) can also be transformed into the vector field (\ref{champ13}) by changing the variables $(v_1,v_2,v_3)$ to $(t_2,t_1,t_3)$. So, we finish the proof of the  Proposition.

\end{proof}

\subsection{Proof of Theorem \ref{theoremA} }
\begin{proof}
Let $\{\pi_i, \M_i,\W_i,\fol_i,\E_i\}$ be a blow-up sequence such that $\M_0=\P^3$ and $\W_i$ is homeomorphic to $\W_{i-1}$ with $\pi_i({\W_i})=\W_{i-1}$ for all $i\ge 1$. From Theorem \ref{theoremB}, there exists $k_0\in \N$ such that $m_{\E_i}(\pi_i^*\fol_{i-1})=0$ for $i\ge k_0$. Hence, $m_{\W_i}(\fol_i)=1$ and $\W_i$ is of type III for some index $i$. From Propostions \ref{proed}, \ref{pro1an}, \ref{pro00} and \ref{pro01}, we conclude the proof of Theorem \ref{theoremA}. Thus, $\W_i$ is an elementary component for $i\ge k$, for some natural number $k$, and for almost all points of $\W_i$. If follows from \cite[Proposition 2.20]{Cascini-Spicer} that $(\fol_i)=1$ is   generically log canonical along $\W_i$. 
\end{proof}
\begin{example}\label{SSe}({\bf F. Sanz and F. Sancho's example}) \rm  Let us consider the holomorphic foliation $\fol_0$ defined on $\M_0=\P^3$
described in the affine open set $U_3=\{[\xi]\in\P_3, \xi_3\ne 0\}$ by
the following vector field
$$X_{0}= \dps  z_1^2\frac{\partial}{\partial z_1}+(-\alpha z_1z_2+z_1z_3)\frac{\partial}{\partial z_2}+(-\lambda z_1+z_2-\beta z_1z_3)\frac{\partial}{\partial z_3},$$
where $z_i=\frac{\xi_{i-1}}{\xi_3}$, $\alpha,\beta\in\R_{\ge0}$ and
$\lambda\in\R_{>0}$. Thus, we have that
$$\sing(\fol_0)=\W_0\cup\W_1\cup \{p_1\}$$
where $\W_0:=\{\xi_0=\xi_1=0\}$, $\W_1:\{\xi_0=\xi_3=0\}$
and $p_1=[1:0:0:0]$. Let $\pi_1:\M_1\to\M_0$ be the blowup of $\P^3$
 along  $\W_0$ being $\E_1$ and $\fol_1$ as in the previous
example. Thus, the curve $\W_0$ is type III. The singular set of $\fol_1$
contains only one curve which is homeomorphic to $\W_0$  but with
multiplicity equal to 2. See Theorem \ref{theoremt}.

In fact, in the chart $((U_1)_2,z=\sigma_2(v))$, with relations $z=\sigma_2(v)=(v_1v_2,v_2,v_3)$, the foliation $\fol_1$ is
described by the following vector field
$$ Y_1=\big(v_2v_1^2-v_1^2(-\alpha v_2+v_3)\big)\frac{\partial}{\partial v_1}+v_2(-\alpha v_1v_2+v_1v_3)\frac{\partial}{\partial v_2}+v_2\big(-v_1(\lambda+\beta v_3)+1\big)\frac{\partial}{\partial v_3}.$$

It is not difficult to see that the curve defined as
$\W_2^{(1)}=\{v\in (U_1)_2|v_1=v_2=0\}$ is such a curve. In the chart
$((U_1)_1,\sigma_1(u))$, the foliation $\fol_1$ is
described by the vector field
$$X_{1}=u_1^2\frac{\partial}{\partial u_1}+(u_3-u_1u_2(1+\alpha))\frac{\partial}{\partial u_2}+u_1\big(-\lambda-\beta v_3 + u_2)\frac{\partial}{\partial u_3}.$$

However, the singular set of $\fol_1$ contains the curve
$\W_1^{(1)}=\{u\in (U_1)_2|u_1=u_3=0\}$ which is  homeomorphic   to $\P^1$. F. Sanz and F. Sancho showed that the vector field $X_1$ is invariant by a blow-up centered at $\W_1^{(1)}$ which results that $X_0$ cannot be
desingularized by blow-ups along such curves. See \cite{FSFS} for more details.

Let $\pi_2:\M_2\to \M_1$ be the blowup of $\M_1$ along  $\W_2^{(1)}$ being $\E_2$ and $\fol_2$ the
exceptional divisor and the strict transform foliation, respectively. Thus,
the curve $\W_2^{(1)}$ is of type I and
$\mbox{m}_{\E_2}(\pi_2^*\fol_1)=1$. In the chart $((U_2)_1,\sigma_1(w)=v)$, the foliation $\fol_2$ is described by
the vector field
\begin{align*}
 X_2=&\dps
w_1\big(-w_3+(1+\alpha)w_1w_2\big)\frac{\partial}{\partial
w_1}+w_2\big(2w_3-(1+2\alpha)w_1w_2\big)\frac{\partial}{\partial
w_2}+\cr & w_2\big(1-w_1(\lambda
+\beta)\big)\frac{\partial}{\partial w_3}.
\end{align*}

Except for $w_3=0$, the curve
$\W_1^{(2)}=\{w\in(U_2)_1|w_1=w_2=0\}$ is an elementary curve of
$\sing(\fol_2)$ as the eigenvalues of $\mathbf{A}_{X_2}|_{\W_1^{(2)}}$ are
$\lambda_{11}^{(2)}(w_3)=-w_3$ and $\lambda_{21}^{(2)}(w_3)=2w_3$.
Furthermore, $\lambda_{11}^{(2)}/\lambda_{21}^{(2)}=-1/2\not\in \Q_+$
for almost all $w\in\W_1^{(1)}$. The exception occurs precisely at the
intersection point with the curve $\pi_2^{-1}\big(\W_1^{(1)}\big)$.

\end{example}

\end{document}